\documentclass{m2an}
\usepackage{etoolbox,tikz,pgfplots,array,hyperref,amssymb,thmtools,xcolor,mathtools,enumitem,float,bm}
\hypersetup{
    colorlinks=true,
    citecolor=blue,
    linkcolor=blue
    }
\pdfpagesattr{ /CropBox [0 40 555 760] }
\usepackage[a4,center]{crop}
\usepackage[capitalise,nameinlink,noabbrev]{cleveref}

\newtheorem{theorem}{Theorem}[section]

\newtheorem{lemma}[theorem]{Lemma}
\newtheorem{definition}[theorem]{Definition}
\newtheorem{assumption}[theorem]{Assumption}
\newtheorem{problem}[theorem]{Problem}

\newtheorem{remark}[theorem]{Remark}

\numberwithin{equation}{section}
\crefname{subsection}{Subsection}{Subsections}
\crefformat{equation}{(#2#1#3)}
\crefformat{align}{(#2#1#3)}
\crefformat{enumi}{(#2#1#3)}
\crefname{figure}{Figure}{Figures}
\newcommand{\ddt}{\frac{\dd}{\dd\mathrm{t}}}

\crefname{theorem}{Theorem}{Theorems}
\crefname{definition}{Definition}{Definitions}
\crefname{lemma}{Lemma}{Lemmas}
\crefname{assumption}{Assumption}{Assumptions}
\crefname{problem}{Problem}{Problems}
\crefname{experiment}{Experiment}{Experiments}
\crefname{remark}{Remark}{Remarks}
\crefname{comment}{Comment}{Comments}

\newcommand{\citep}{\cite}

\DeclarePairedDelimiter{\norm}{\|}{\|}

\crefname{subsection}{Subsection}{Subsections}
\crefformat{equation}{(#2#1#3)}
\crefformat{align}{(#2#1#3)}
\crefformat{enumi}{(#2#1#3)}
\crefname{figure}{Figure}{Figures}

\newcommand{\ff}{\bm{f}}
\newcommand{\OT}{{\Omega_T}}
\renewcommand{\O}{{\Omega}}

\renewcommand{\rho}{\varrho}

\def\Th{\mathcal{T}_h}
\def\Itau{\mathcal{I}_\tau}

\def\phm{\phantom{$-$}}
\def\la{(}
\def\ra{)}

\def\Vh{\mathcal{V}_h}
\def\Qh{\mathcal{Q}_h}
\def\Xh{\mathcal{X}_h}
\def\dtau{d^{n+1}_\tau}
\def\phm{\phantom{-}}
\DeclarePairedDelimiter{\snorm}{|}{|}
\newcommand\restr[2]{{% we make the whole thing an ordinary symbol
		\left.\kern-\nulldelimiterspace % automatically resize the bar with \right
		#1 % the function
		\vphantom{\big|} % pretend it's a little taller at normal size
		\right|_{#2} % this is the delimiter
}}
\newcommand\restrb[3]{{% we make the whole thing an ordinary symbol
		\left.\kern-\nulldelimiterspace % automatically resize the bar with \right
		#1 % the function
		\vphantom{\big|} % pretend it's a little taller at normal size
		\right|_{#2}^{#3} % this is the delimiter
}}

\newcommand{\eps}{\varepsilon}

\newcommand{\con}{\hookrightarrow}
\newcommand{\com}{%
	\mathrel{\mathrlap{{\mspace{4mu}\lhook}}{\hookrightarrow}}}

\newcommand{\dd}{\textup{d}}
\newcommand{\dx}{\,\textup{d}x}

\newcommand{\dt}{\,\textup{d}t}
\newcommand{\ds}{\,\textup{d}s}

\renewcommand{\ddt}{\frac{\dd}{\dd t}}

\newcommand{\R}{\mathbb{R}}
\newcommand{\N}{\mathbb{N}}
\newcommand{\p}{\partial}
\newcommand{\pt}{\p_t}

\newcommand{\num}{\textup{num}}
\newcommand{\av}{\textup{av}}

\renewcommand{\div}{\textup{div}}

\DeclareFontFamily{U}{bbold}{}
\DeclareFontShape{U}{bbold}{m}{n}
{
	<-5.5> s*[1.069] bbold5
	<5.5-6.5> s*[1.069] bbold6
	<6.5-7.5> s*[1.069] bbold7
	<7.5-8.5> s*[1.069] bbold8
	<8.5-9.5> s*[1.069] bbold9
	<9.5-11> s*[1.069] bbold10
	<11-15> s*[1.069] bbold12
	<15-> s*[1.069] bbold17
}{}

\def\signed #1{{\leavevmode\unskip\nobreak\hfil\penalty50\hskip2em
		\hbox{}\nobreak\hfil(#1)%
		\parfillskip=0pt \finalhyphendemerits=0 \endgraf}}

\newsavebox\mybox

\def\vv{{\bm{v}}}

\def\ww{\bm{w}}
\def\u{\bm{u}}

\renewcommand{\bar}[1]{\mkern 1.5mu\overline{\mkern-1.5mu#1\mkern-1.5mu}\mkern 1.5mu}

\newcommand{\nvec}{\bm{n}}

\newcommand{\Div}{\ensuremath{\mathrm{div}}}

\usepackage[misc]{ifsym}
\mathcode`\;="603B
\mathcode`\,="613B
\newcolumntype{M}[1]{>{\centering\arraybackslash}m{#1}}
\allowdisplaybreaks
\begin{document}
%%-----------------------------
%%      the top matter
%%-----------------------------
\title{Analysis and structure-preserving approximation of a Cahn--Hilliard--Forchheimer system with solution-dependent mass and volume source}
\author{Aaron Brunk}\address{Institute of Mathematics, Johannes Gutenberg University, Mainz, Germany}
\author{Marvin Fritz}\address{Computational Methods for PDEs, Johann Radon Institute for Computational and Applied Mathematics, Linz, Austria}
%
%\date{\vspace{-1cm}}
%
\begin{abstract} We analyze a coupled Cahn--Hilliard--Forchheimer system featuring concentration-dependent mobility, mass source and convective transport. The velocity field is governed by a generalized quasi-incompressible Forchheimer equation with solution-dependent volume source. We impose Dirichlet boundary conditions for the pressure to accommodate the source term. Our contributions include a novel well-posedness result for the generalized Forchheimer subsystem via the Browder--Minty theorem, and existence of weak solutions for the full coupled system established through energy estimates at the Galerkin level combined with compactness techniques such as Aubin--Lions' lemma and Minty's trick. Furthermore, we develop a structure-preserving discretization using Raviart--Thomas elements for the velocity that maintains exact mass balance and discrete energy-dissipation balance, with well-posedness demonstrated through relative energy estimates and inf-sup stability. Lastly, we validate our model through numerical experiments, demonstrating optimal convergence rates, structure preservation, and the role of the Forchheimer nonlinearity in governing phase-field evolution dynamics. \end{abstract}
\date{\hphantom{.}}
\mbox{}\vspace{-1cm}

\maketitle

\section{Introduction}
The mathematical modeling of tumor growth has evolved significantly, with the foundational “four-species model'' accounting for the volume fractions of cancer cells, healthy cells, and nutrient-rich and nutrient-poor extracellular water \cite{hawkins2012numerical}. The mathematical properties of this model, particularly its well-posedness, have been extensively studied \cite{garcke2017analysis,garcke2017well}. In tumor growth modeling, cells are typically represented as viscous, inertia-free fluids, with the fluid mixture's velocity modeled in a volume-averaged sense, an appropriate assumption given the dense packing of cells. Various approaches have been developed to incorporate fluid movement, such as Darcy's law  \cite{garcke2016global,garcke2016cahn,garcke2018cahn} and the Brinkman equation \cite{ebenbeck2019analysis,ebenbeck2019cahn}.
%, and the unsteady Forchheimer--Brinkman equation \cite{fritz2019unsteady}. 

The Forchheimer equation itself has been extensively analyzed \cite{audu2018well,pan2012mixed,hu2023well,knabner2014mathematical,knabner2016sovability,fabrie1989regularity}. Forchheimer's law, as reviewed in \cite{whitaker1996forchheimer}, represents an advancement over Darcy's law by accounting for quadratic behavior in the pressure drop versus flow rate relationship. This equation can be derived through mixture theory \cite{srinivasan2014thermodynamic} and is particularly relevant for high-velocity non-Darcy flow scenarios.

Our work advances the field in several key aspects. While previous studies coupled the Cahn--Hilliard equation with flow models, we demonstrate that the time-differential term and the Brinkman contribution in the unsteady Forchheimer--Brinkman formulation as in \cite{fritz2019unsteady} are analytically unnecessary through a monotone operator approach and Minty's trick. We extend the analysis to the quasi-incompressible case and generalize the nonlinear term to $|u|^{s-1}u$ for $s>1$, encompassing the typical Forchheimer law ($s=2$) as a special case. While proving the existence of weak solutions to the Cahn--Hilliard--Forchheimer system, we derive a novel well-posedness result to the sole Forchheimer equation with natural boundary conditions.

We develop a fully discrete scheme using Raviart--Thomas elements for velocity approximation, discontinuous linear elements for the pressure, and continuous linear elements for the phase-field and chemical potential, inheriting essential geometric structures from the continuum model. Following the framework for structure-preserving PDE discretizations \cite{Egger2019StructureProblems}, our scheme exactly preserves the mass balance through compatible flux discretization and maintains a discrete energy-dissipation balance via careful treatment of nonlinear coupling terms. This extends recent advances in energy-stable phase-field simulations \cite{brunk2025structure-preserving,brunk2023second-order,brunk2024variational} to quasi-incompressible flow regimes. Building on inf-sup stable mixed formulations, we establish the unique solvability of our nonlinear discrete system through relative energy estimates similar to \cite{brunk2023stability}.

The outline of the article is as follows: In \cref{Sec:Prelim}, we establish the notation and present the coupled Cahn--Hilliard--Forchheimer system, including references to foundational Forchheimer analyses. A refined well-posedness result for the decoupled Forchheimer subsystem is detailed in \cref{Appendix}. \cref{Sec:Analysis} develops the existence theory for weak solutions through a Galerkin approximation scheme with energy estimates, emphasizing the challenges of nonlinear coupling. \cref{Sec:Appro} introduces the structure-preserving discrete formulation using Raviart--Thomas elements, proving: (i) exact mass balance, (ii) energy-dissipation balance, and (iii) unique solvability via relative energy estimates. Finally, \cref{Sec:Simul} demonstrates numerical simulations quantifying Forchheimer's nonlinear effects on the phase-field's morphology.

\section{Preliminaries} \label{Sec:Prelim}

\subsection{Notation}
We consider a bounded domain $\Omega \subset \mathbb{R}^d$, where $d\in\{2,3\}$, and a fixed time $T > 0$. Furthermore, the space-time domain is denoted by $\Omega_T:= \Omega \times (0,T)$, and we use the notation $(f,g)_\Omega=\int_\Omega f g \dx$ and denote the mean of $f:\Omega \to \R$ by $f_\Omega$, that is, $f_\Omega=\frac{1}{|\O|} (1,f)_\Omega$. Similarly, we denote $(f,g)_{\partial\Omega}=\int_{\partial\Omega} f g \mathrm{d}S$, $(f,g)_{(0,T)}=\int_{(0,T)} f g \;\mathrm{d}t$ and $(f,g)_{\Omega_T}=\int_{\Omega\times(0,T)} f g \dx\mathrm{d}t$. The dual pairing between two $X$ and $X'$ is denoted by $\langle a,b\rangle_{X\times X'}$ and the subscript can be neglected if the spaces for pairings are clear.
We write $x\lesssim y$ for $x\leq Cy$ when the generic constant $C$ does not contain any relevant parameters.

\subsection{Modeling}

In the following, we consider a Cahn--Hilliard--Forchheimer system in tumor growth, that is,	
\begin{equation}\label{Eq:CHF}\begin{aligned}
		\div\, \vv &= \Gamma_\vv(\phi), \\
		\nabla p + \alpha(\phi) \vv + \beta(\phi)|\vv|^{s-1} \vv  &= \mu\nabla \phi, \\
		\pt \phi + \div(\phi \vv) &= \div(m(\phi)\nabla \mu) + \Gamma_\phi(\phi), \\
		\mu &=  \Psi'(\phi)-\eps^2 \Delta \phi,
\end{aligned}\end{equation}
where $\vv$ denotes the volume-averaged velocity of the mixture, $p$  the associated pressure, $\phi \in [0, 1]$ the difference in volume fractions, with $\{x:\phi(x) = 1\}$ representing the unmixed tumor tissue, and $\{x:\phi(x) = 0\}$ representing the surrounding healthy tissue, and finally, $\mu$ denotes the chemical potential for $\phi$. The constant $\eps > 0$ is related to the thickness of the diffuse interface. The function $m(\cdot)$ is the mobility of the phase variable $\phi$, $\alpha(\cdot)$ the tumor-dependent permeability of the fluid, $\beta(\cdot)$ the Forchheimer influence of the fluid, $\Gamma_\vv(\cdot)$ the compressibility of the fluid and $\Gamma_\phi(\cdot)$ the source function of the tumor variable.

In this work we consider the follow initial and boundary conditions:
\begin{equation}\label{eq:boundary_condition}
	\begin{aligned}
		[(1-\gamma)\vv - \nabla \mu] \cdot \nvec &= \nabla \phi \cdot \bm{n}= 0 &&\text{ on } \partial \Omega \times (0, T), \quad \gamma\in\{0,1\}, \\
		p &= 0 &&\text{ on } \partial \Omega \times (0, T), \\
		\phi(0) &= \phi_0 &&\text{ in } \Omega.
	\end{aligned}
\end{equation}
Alternatively, the boundary condition on $p$ can be replaced by the Neumann boundary condition for $v$, that is, $\vv \cdot \bm{n}= 0$ on $\partial\Omega$. This implies a compatibility condition $(\Gamma_\vv,1)_\Omega=0$, which a solution-dependent source term $\Gamma_\vv$ cannot, in
general, fulfill.
The Dirichlet boundary condition on $p$ appears as a natural condition and the Neumann boundary condition on $\vv$ appears as an essential condition here. It is contrary to the primal mixed element, under which the Dirichlet boundary condition seems to be an essential condition and the Neumann boundary condition seems as a natural condition.

We emphasize that in the proof of existence for weak solutions, we consider the case of $\gamma=1$, while the numerical scheme and all related results hold for both cases, that is, $\gamma\in\{0,1\}$. Hence, it is worth discussing the difference between both boundary conditions. The main difference of both boundary conditions manifest in the mass and energy-dissipation balance equation, which reads
\begin{align*}
	\ddt\la \phi,1 \ra &= \la \Gamma_\phi(\phi),1 \ra - \gamma\la \phi,\vv\cdot\nvec\ra_{\partial\Omega}, \\
	\ddt E(\phi) &+ \mathcal{D}_\phi(\mu,\u) = \la \Gamma_\phi(\phi),\mu \ra + \la p,\div(\vv) \ra - \gamma\la \phi\mu,\div(\vv) \ra,
\end{align*}
where the energy and dissipation is given by
\begin{align*}
	\mathcal{E}(\phi):=\frac{\varepsilon^2}{2}\norm{\nabla\phi}_{L^2}^2 + (\Psi(\phi),1)_\Omega, \quad \mathcal{D}_\phi(\mu,\u) := \la m(\phi)\nabla\mu,\nabla\mu \ra + \la (\alpha(\phi)+\beta(\phi)\snorm{\vv}^{s-1})\vv,\vv \ra. 
\end{align*}

We can clearly see that $\gamma=0$ implies that there is no mass flow through the boundary and also no additional force in the energy balance. Although $\gamma=1$ allows for both effects. In the application of tumor growth, the first would restrict to growth of already existing tumors in the domain, while the second would also allow for tumor cells to flow through the boundary.
The drawback of course is that the mass can increase and the order parameter $\phi$ might leave the desired range, i.e., breaking the application context.
We may note that both are compatible with the principle of energy dissipation in absence of forces, i.e. $\Gamma_\phi=\Gamma_\vv=0$, but mass conservation is only given for $\gamma=0.$

We may note that these boundary conditions allow for certain pitfalls. Let us suppose that $\phi=\textrm{const}$ on the whole boundary $\partial\Omega$ and also a small boundary layer around. In this case $\nabla\mu=0$ and the boundary conditions implies
\begin{equation*}
	0 =  \int_{\partial\Omega} \vv\cdot\nvec  = \int_\Omega \div(\vv)= \int_\Omega \Gamma_\vv(\phi).  
\end{equation*}
However, this is in general not true and again imposes certain compatibility conditions on $\Gamma_v(\phi).$ 

\begin{remark} One can consider nutrient evolution and cell attractant effects such as chemotaxis, as done in \cite{fritz2023tumor,fritz2021analysis,fritz2019local}. Then the model can be further extended by an equation for the nutrient density $\sigma$, which couples to the other equations as follows:
	\begin{equation}\label{eq:system}\begin{aligned}
			\div\, \vv &= \Gamma_\vv(\phi,\sigma), \\
			\nabla p + \alpha(\phi) \vv + \beta(\phi) |\vv|^{s-1} \vv  &= (\mu+\chi\sigma)\nabla \phi, \\
			\pt \phi + \div(\phi \vv) &= \div(m(\phi)\nabla \mu) + \Gamma_\phi(\phi,\sigma), \\
			\mu &=  \Psi'(\phi)-\eps^2 \Delta \phi - \chi \sigma, \\
			\pt \sigma + \div(\sigma \vv) &= \div(n(\sigma)\nabla (D\sigma-\chi\phi)) + S_\sigma(\phi,\sigma).
	\end{aligned}\end{equation}
	Here, the nonnegative parameter $\chi$ represents chemotaxis, that is, the attraction between nutrients and tumor cells. The existence of the nutrient-extended system can be shown in the same manner as for the Cahn--Hilliard--Forchheimer system in this work. Our model can be seen as a special case of this nutrient-extended system considering $\chi=0$ and the continuous presence of nutrients $\sigma \equiv 1$.
\end{remark}

\subsection{Forchheimer equation}
We extend the existing literature on the generalized Forchheimer equation in several significant ways. Although previous studies \cite{audu2018well,fabrie1989regularity} focused primarily on the case with a boundary condition of $\vv$ and imposed compatibility conditions on the right-hand side of the divergence equation, we examine the equation with pressure boundary conditions to allow solution-dependent source terms, which is relevant in the application of a tumor growth model. In fact, we study the system
\begin{equation}\label{Eq:Forchheimer}\begin{aligned}
		\beta|\vv|^{s-1} \vv+ \alpha \vv + \nabla p &= \bm{f}&&\text{in } \Omega, \\
		\div \vv &= g &&\text{in } \Omega, \\
		p &= h &&\text{on } \partial\Omega.
\end{aligned}\end{equation}
The pressure boundary case had previously been analyzed only for the case $s=2$ in \cite{knabner2014mathematical,knabner2016sovability} or for $\div \vv \in L^2(\Omega)$ in \cite{kieu2020existence}, but our work generalizes this to arbitrary $s>1$ and $\div \vv=g \in L^{1+s}(\Omega)$. This extension bridges the gap between the classical pressure boundary analysis and the broader $s$-case with velocity boundary conditions.
\begin{figure}[H] \centering 
	\begin{tabular}{ |c||c|c|c|c|c|  }
		\hline
		Ref & $s$ & Boundary & $g$ & $\bm{f}$ & $(\vv,p)$ \\
		\hline
		\cite{fabrie1989regularity} & $(1,\infty)$ & $\vv \cdot \bm{n}|_{\p\Omega} \in W^{s/(1+s),1+s}$  & $L^{1+s}$ & $L^{1+1/s}$ & $L^{1+s} \times W^{1,1+1/s}$ \\  
		\cite{audu2018well} & $(1,2]$ &  $\vv \cdot \bm{n}|_{\p\Omega} \in L^{(1+s)(1-1/d)}$ & $L^{d(1+s)/(d+1+s)}$ & $L^{(1+s)/s}$ & $L^{1+s} \times W^{1,(1+s)/s}$ \\
		\cite{girault2008numerical} & $2$ & $\vv \cdot \bm{n}|_{\p\Omega} \in L^{3(d-1)/d} $ & $L^{3d/(d+3)}$ & $0$ & $L^3 \times W^{1,3/2}$
		\\ \hline 
		\cite{pan2012mixed} & $2$ & $p|_{\p\Omega} \in L^2$ & $L^2$ & $L^2$ & $L^3 \cap L^2_\div \times L^2$ \\ 
		\cite{knabner2016sovability} & $2$ & $p|_{\p\Omega} \in W^{1/3,3/2}$ & $L^3$ & $0$ & $L^3_\div \times L^{3/2}$ \\ 
		\cite{kieu2020existence} & $(1,\infty)$ & $p|_{\p\Omega} \in W^{1/s,s}$ & $L^2$ & $0$ & $L^{1+s} \cap L^2_\div \times L^2$ \\
		here & $(1,\infty)$ & $p|_{\p\Omega}\in W^{1/(s+1),1+1/s}$ & $L^{1+s}$ & $L^{1+1/s}$ & $L^{1+s}_\div \times W^{1+1/s}$ \\ \hline
	\end{tabular}
	\caption{Comparison of results for the generalized compressible Forchheimer system $\beta|\vv|^{s-1}\vv+\alpha \vv+\nabla p = \bm{f}$ and $\div \vv = g$.}
\end{figure}

Naturally, the spatial solution space of a velocity $\vv$, that is governed by Forchheimer's law, is 
$$L^{1+s}_\div(\Omega)=\{\u \in L^{1+s}(\Omega)^d: \div \u \in L^{1+s}(\Omega) \}$$
Since $L^{1+s}_\div(\Omega)$ is a closed subspace of $L^{1+s}(\Omega)^d$, it follows that $L^{1+s}_\div(\Omega)$ is a reflexive Banach space.

\begin{lemma} \label{Lem:Forchheimer} We make the following assumptions:
	\begin{itemize}\itemsep.1em
		\item $\Omega \subset \mathbb{R}^d$, $d \in \{2,3\}$, bounded domain with $C^2$-boundary,
		\item $s\in(1,\infty)$,
		\item $\ff\in L^{1+1/s}(\Omega)^d$, 
		\item $g \in L^{1+s}(\Omega)$, 
		\item $h \in W^{1/(s+1),1+1/s}(\p\Omega)$, \item $\beta,\alpha \in L^\infty(\R)$ with $\beta(x) \geq \beta_0$, $\alpha(x) \geq \alpha_0$ for all $x \in \R$.
	\end{itemize}
	Then there exists a unique weak solution $(\vv, p) \in L_\div^{1+s}(\Omega)^d \times W^{1+1/s}(\Omega)$ to the Forchheimer system \cref{Eq:Forchheimer} 
	satisfying the weak form \cref{mix.form.stat.pr} and the following estimate:
	\begin{equation}
		\|\vv\|_{L_\div^{1+s}} + \|p\|_{W^{1+1/s}} \lesssim \|\ff\|_{L^{1+1/s}} + \|g\|_{L^{1+s}} + \|h\|_{W^{1/(s+1),1+1/s}(\p\Omega)}.
	\end{equation}
\end{lemma}
The proof is carried out in \cref{Appendix}. Furthermore, we have the following well-known result for the divergence equation: 
\begin{lemma}[cf.~{\cite[III.3]{galdi2011introduction}}] \label{Lem:Divergence} Let $\Omega \subset \R^d$, $d\geq 2$, be a bounded Lipschitz domain. Then, for every $g \in L^q(\Omega)$, $q \in (1,\infty)$, and $\bm{a} \in W^{1-1/q,q}(\p\Omega)^d$ with $(g,1)_\Omega=(\bm{a},\bm{n})_{\p\Omega}$, there exists a solution $\u \in W^{1,q}(\Omega)$ of the problem
	$$\begin{aligned}
		\div \u&= g &&\text{ in } \Omega, \\
		\u &= \bm{a} &&\text{ on } \p\Omega.
	\end{aligned}$$
	Additionally, it holds
	$$\|\u\|_{W^{1,q}} \lesssim \|g\|_{L^q} + \|\bm{a}\|_{W^{1-1/q,q}(\p\Omega)}.$$
\end{lemma}

\section{Existence of weak solutions} \label{Sec:Analysis}

We make the following assumptions:
\begin{assumption} \label{Assumptions} ~\\[-0.3cm]
	\begin{enumerate}[label=\textup{(A\arabic*)}, ref=A\arabic*, leftmargin=1.2cm] \itemsep.1em
		\item  $\eps>0$ is fixed, $s>1$, $\phi_0\in H^1(\Omega)$,
		\item $m,\alpha,\beta \in C^0(\R;\R_{>0})$ satisfy $
		x_0\leq x(t)\leq x_\infty$ with $x_0,x_\infty>0$ where $x \in \{m,\alpha,\beta\}$.
		\item $\Gamma_\vv,\Gamma_\phi \in C^0(\R)\cap L^\infty(\R)$.
		\item $\Psi\in C^2(\R,\R_{\geq 0})$ fulfills $ s^2 - 1\lesssim \Psi(s)\lesssim s^2 +1$, $|\Psi'(s)|\lesssim 1+|s|$ and
		$|\Psi''(s)|\lesssim 1+|s|^q$ for some $q \in [0,4).$
	\end{enumerate}
\end{assumption}

We note that the case of $s=1$ reduced the Cahn--Hilliard--Forchheimer system to the Cahn--Hilliard--Darcy system with mass source that has been studied in \cite{sprekels2021optimal,giorgini2022existence,jiang2015well,garcke2018cahn}.
The assumption on $\Psi$ is fulfilled by the typical quartic potential that is quadratically continued outside of the relevant interval $[0,1]$. This assumption can certainly be further loosened, but this is not the focus of the work. For works on Cahn--Hilliard systems with singular potentials, we refer to \cite{ebenbeck2019cahn,frigeri2017multi}. We now introduce the weak formulation of the Cahn--Hilliard--Forchheimer system \cref{Eq:CHF}.

\begin{definition}[Weak solution of \cref{Eq:CHF}]\label{Def:WeakSol}We call a quadruple $(\phi,\mu,\vv,p)$ a weak solution to \cref{Eq:CHF} if 
	\begin{align*}
		\phi &\in H^1(0,T;(H^1(\Omega))')\cap L^{\infty}(0,T;H^1(\Omega))\cap L^2(0,T;H^3(\Omega)), \\ 
		\mu &\in L^4(0,T;L^2(\Omega)) \cap L^2(0,T;H^1(\Omega)),\\
		\vv&\in L^{s+1}(0,T;L_\div^{s+1}(\Omega)^d),\\ 
		p&\in L^{\sigma'}(0,T;W_0^{1+1/s}(\Omega)), \quad \sigma'>1,
	\end{align*}
	such that 
	\begin{equation*}
		\div(\vv) = \Gamma_\vv(\phi)\text{ a.e. in }\Omega_T, \quad \phi(0)=\phi_0\text{ a.e. in }\Omega,
	\end{equation*}
	and 
	\begin{subequations}
		\label{WFORM_1}
		\begin{align}
			\label{WFORM_1a}0 &= (\beta(\phi)|\vv|^{s-1}\vv,\bm{\xi})_\Omega + (\alpha(\phi) \vv,\bm{\xi})_\Omega   -  (\mu\nabla\phi,\bm{\xi})_\Omega - (p, \div\bm{\xi})_\Omega   ,\\
			\label{WFORM_1b} 0 &= \langle\pt\phi{,}\zeta\rangle  + (m(\phi)\nabla\mu,\nabla\zeta)_\Omega-(\Gamma_\phi(\phi),\zeta)_\Omega + (\nabla \phi \cdot \vv,\zeta)_\Omega + (\phi \Gamma_\vv(\phi),\zeta)_\Omega,   \\
			\label{WFORM_1c} 0 &= -(\mu,\zeta)_\Omega +(\Psi'(\phi),\zeta)_\Omega + \eps^2 (\nabla\phi,\nabla\zeta)_\Omega,
		\end{align}
	\end{subequations}
	for a.e. $t\in(0,T)$ and for all $\bm{\xi}\in L^{s+1}_\div(\Omega)^d$, $\zeta \in H^1(\Omega)$.
\end{definition}

\begin{remark}
	Although not presented here, the proof for the other boundary condition, that is $\gamma=0$, follows the same lines.    
\end{remark}

\begin{theorem}\label{THM_WSOL_1}
	Let $\Omega\subset\R^d,\,d\in\{2,3\},$ be a bounded domain with $C^3$-boundary and assume that \cref{Assumptions} is fulfilled and $\gamma=1$. Then there exists a solution quadruple $(\phi,\mu,\vv,p)$ of \cref{Eq:CHF} in the sense of  \cref{Def:WeakSol}. Furthermore, we have $\phi \in L^4(0,T;H^2(\Omega))\cap L^2(0,T;H^3(\Omega))$ 
	and $\mu \in L^4(0,T;L^2(\Omega)$. In addition the integrated energy inequality holds for almost all $t\in(0,T)$, i.e.
	\begin{equation}\label{Eq:Energy}
		\begin{aligned}
			\mathcal{E}(\phi)(t) + \int_0^t\mathcal{D}(\mu,\vv) \leq \mathcal{E}(\phi)(0) + \int_0^t(\mu,\Gamma_{\phi})_\Omega + (p,\Gamma_{\vv})_\Omega -  (\phi\mu,\Gamma_{\vv})_\Omega
	\end{aligned}\end{equation}
	with dissipation
	\begin{equation*}
		\mathcal{D}(\mu_k,\vv_k):= (m(\phi_k),|\nabla\mu_k|^2)_\Omega +(\beta_k|v_k|^{s-1} \vv_k,\vv_k)_\Omega +  (\alpha_k \vv_k,\vv_k)_\Omega  .
	\end{equation*} 
\end{theorem}

\begin{proof} We employ the Galerkin method to reduce the system to ODEs, derive energy estimates, and pass to the limit using compactness arguments. \medskip
	
	\noindent\textbf{Step 1: Galerkin approximation.}
	We will construct approximate solutions by applying a Galerkin approximation with respect to $(\phi,\mu)$ and at the same time solve for $v$ and $p$ in the corresponding whole function spaces. As a Galerkin basis for $(\phi,\mu)$, we use the eigenfunctions of the Neumann--Laplace operator $\{w_i\}_{i\in\N}$ that form an orthonormal Schauder basis in $L^2(\Omega)$ which is also a basis of $H_N^2(\Omega)$. 
	We fix $k\in\N$ and define ${W}_k= \text{span}\{w_1,...,w_k\}$.
	Our aim is to find functions of the form 
	\begin{equation*}
		\phi_k(t,x)=\sum_{i=1}^{k}a_i^k(t)w_i(x),\qquad \mu_k(t,x)=\sum_{i=1}^{k}b_i^k(t)w_i(x),
	\end{equation*}
	satisfying the following approximate problem:
	\begin{subequations}
		\begin{align}
			\label{Eq:GalerkinPhi}(\p_t\phi_k,\xi_k)_\Omega &= (-m(\phi_k)\nabla\mu_k,\nabla \xi_k)_\Omega +(\Gamma_{\phi,k},\xi_k)_\Omega  -(\nabla\phi_k \cdot \vv_k,\xi_k)_\Omega+(\phi_k\Gamma_{\vv,k},\xi_k)_\Omega,\\
			\label{Eq:GalerkinMu}
			(\mu_k,\xi_k)_\Omega &= \eps (\nabla \phi_k,\nabla \xi_k)_\Omega + (\Psi'(\phi_k),\xi_k)_\Omega,
		\end{align}
		for all $\xi_k\in{W}_k$ where we introduced $\Gamma_{\phi,k}= \Gamma_{\phi}(\phi_k)$ and $\Gamma_{\vv,k}= \Gamma_\vv(\phi_k)$. We define the approximate velocity $\vv_k$ and the approximate pressure $p_k$ as solutions of the generalized Forchheimer system \cref{Eq:Forchheimer} as studied in \cref{Lem:Forchheimer} with $\ff = \mu_k\nabla\phi_k$ and $g = \Gamma_{\vv,k}$.
		Using the continuous embedding $H_N^2(\Omega)\hookrightarrow L^{\infty}(\Omega)$, straightforward arguments yield that
		$\mu_k\nabla\phi_k \in L^2(\Omega)^d \con L^{1+1/s}(\O)^d$ and $\Gamma_{\vv,k}\in H^1(\Omega)\cap L^{\infty}(\Omega)$.
		Therefore, by \cref{Lem:Forchheimer}, we obtain that $(\vv_k,p_k)\in L^{s+1}_\div(\Omega)^d \times W_0^{1+1/s}(\Omega)$ and the following equations are satisfied 
		\begin{alignat}{3}
			\label{Eq:GalerkinP}\nabla p_k + \beta_k|\vv_k|^{s-1}\vv_k +\alpha_k \vv_k &= \mu_k\nabla \phi_k &&\quad\text{a.e. in }\Omega,\\
			\label{Eq:GalerkinV}\div \vv_k &= \Gamma_{\vv,k}&&\quad\text{a.e. in }\Omega,\\
			\nonumber p_k &= {0}&&\quad\text{a.e. on }\p\Omega,
		\end{alignat}
	\end{subequations}
	where $\beta_k=\beta(\phi_k)$ and $\alpha_k=\alpha(\phi_k)$.
	We obtain a coupled system of ODEs with continuously dependent right-hand. The Cauchy--Peano theorem ensures that there exists $T_k\in (0,\infty]$ such that the system has one solution tuple $(\phi_k,\mu_k)\in C^1([0,T_k);{W}_k)^2$. Furthermore, we can define $\vv_k$ and $p_k$ as the solutions of \cref{Eq:GalerkinP} \& \cref{Eq:GalerkinV}. With similar arguments as above, we deduce $(\vv_k(t),p_k(t))\in L_\div^{s+1}(\Omega)^d \times W_0^{1,1+1/s}(\Omega)$ for all $t\in [0,T_k)$. \bigskip

	\noindent\textbf{Step 2: First energy estimate.}
	In order to obtain a-priori estimates, we derive suitable energy estimates. We choose $\xi_k=\mu_k$ in \cref{Eq:GalerkinPhi}, $\xi_k=\pt \phi_k$ in \cref{Eq:GalerkinMu} 
	to obtain
	\begin{equation}\label{Eq:Test1}
		\begin{aligned}
			&\ddt \left( \|\Psi(\phi_k)\|_{L^1}+\frac{\eps^2}{2}\|\nabla\phi_k\|^2_{L^2} \right)  + (m(\phi_k),|\nabla\mu_k|^2)_\Omega =  (\Gamma_{\phi,k},\mu_k)_\Omega -  (\nabla \phi_k\cdot \vv_k + \phi_k \Gamma_{\vv,k},\mu_k)_\Omega. 
	\end{aligned}\end{equation}
	Regarding the Forchheimer system, we use the method of 'subtracting the divergence' as in \cite{ebenbeck2019analysis,ebenbeck2019cahn}. 
	Due to the assumption on $\Gamma_\vv$ (in particular $\Gamma_{\vv,k}\in L^{\infty}$ for all $k\in\N$) and using \cref{Lem:Divergence}, there exists a solution ${\u}_k\in {W}^{1,q}(\Omega)^d$, $q\in(1,\infty)$,  of the problem
	\begin{alignat*}{3}
		\div{\u}_k &= \Gamma_{\vv,k}&& \quad\text{in }\Omega,\\
		{\u}_k &= \frac{1}{|\p\Omega|} (\Gamma_{\vv,k},1)_\Omega \bm{n} \eqqcolon \bm{a}_k&&\quad\text{on }\p\Omega,
	\end{alignat*}
	satisfying for every $q\in (3,\infty)$  the estimate
	\begin{equation}
		\label{divergence_equation_regularity}\norm{{\u}_k}_{L^\infty} \lesssim \norm{{\u}_k}_{W^{1,q}}\lesssim \norm{\Gamma_{\vv,k}}_{L^q}\lesssim \norm{\Gamma_{\vv,k}}_{L^\infty} \leq C,
	\end{equation}
	where we used the continuous embedding $W^{1,q}(\Omega)^d\hookrightarrow L^{\infty}(\Omega)^d$, $q\in (3,\infty)$.
	We remark that the compatibility condition of \cref{Lem:Divergence} is fulfilled since
	\begin{equation*}
		(\bm{a}_k,\bm{n})_{\p\Omega} = \frac{1}{|\p\Omega|}(\Gamma_{\vv,k},1)_\Omega (\bm{n},\bm{n})_{\p\Omega} =  (\Gamma_{\vv,k},1)_\Omega.
	\end{equation*}
	
	Multiplying \cref{Eq:GalerkinP} with $\vv_k-{\u}_k$, integrating over $\Omega$ and by parts, we end up at
	$$\begin{aligned}
		(p_k,\div(\u_k-\vv_k))_\Omega+(\beta_k|v_k|^{s-1} \vv_k,\vv_k-\u_k)_\Omega +  (\alpha_k \vv_k,\vv_k-\u_k)_\Omega= (\mu_k\nabla\phi_k,\vv_k-\u_k)_\Omega.
	\end{aligned}$$
	As $\div \vv_k=\div \u_k$, the first term is zero and we obtain
	\begin{equation} \label{Eq:Test2} \begin{aligned}
			\beta_0 \|\vv_k\|_{L^{s+1}}^{s+1} +\alpha_0\|\vv_k\|^2_{L^2} &=  (\beta_k|\vv_k|^{s-1} \vv_k,\u_k)_\Omega + (\alpha_k \vv_k, {\u}_k)_\Omega +  (\mu_k\nabla\phi_k , \vv_k-\u_k)_\Omega
	\end{aligned}\end{equation}
	Summing \cref{Eq:Test1} and \cref{Eq:Test2} gives
	\begin{equation} \label{energy_identity_3}\begin{aligned}
			&\ddt \left(  \|\Psi(\phi_k)\|_{L^1}+\frac{\eps^2}{2}\|\nabla\phi_k\|_{L^2}^2\right) +  \beta_0 \|\vv_k\|_{L^{s+1}}^{s+1}+\alpha_0\|\vv_k\|^2_{L^2} + m_0\|\nabla\mu_k\|^2 \\
			&\leq  (\Gamma_{\phi,k},\mu_k)_\Omega -  (\nabla \phi_k\cdot{\u}_k + \phi_k \Gamma_{\vv,k},\mu_k)_\Omega + (\beta_k|v_k|^{s-1} \vv_k,\u_k)_\Omega + (\alpha_k \vv_k,{\u}_k)_\Omega.
	\end{aligned}\end{equation}
	Using Hölder's and Young's inequalities, we use \cref{divergence_equation_regularity} for the two last terms on the right-hand side to get
	\begin{equation} \label{A_priori_Stokes}\begin{aligned}
			(\alpha_k{\vv}_k,{\u}_k)_\Omega&\leq \alpha_\infty \norm{\vv_k}_{{L}^2}\norm{{\u}_k}_{{L}^2}
			\leq  \frac{\alpha_0}{2}\norm{\vv_k}_{{L}^2}^2 + C. \\
			(\beta_k|{\vv}_k|^{s-1}\vv_k,{\u}_k)_\Omega&\leq \beta_\infty \norm{\vv_k}_{{L}^{s+1}}^s\norm{{\u}_k}_{{L}^{s+1}}
			\leq  \frac{\beta_0}{2}\norm{\vv_k}_{{L}^{s+1}}^{s+1} + C.
	\end{aligned}\end{equation}
	Next, we want to deduce an estimate for the $L^2(\Omega)$-norm of $\mu_k$. Inserting $\xi_k=\mu_k$ into \cref{Eq:GalerkinMu}  yields
	\begin{equation*}
		\|\mu_k\|^2_{L^2} =  (\Psi'(\phi_k),\mu_k)_\Omega + \eps^2(\nabla \phi_k,\nabla\mu_k)_\Omega .
	\end{equation*} 
	Using Hölder's and Young's inequalities together with the assumptions on $\Psi$, see (A4) in \cref{Assumptions}, we obtain
	\begin{align*}
		\norm{\mu_k}_{L^2}^2 &\leq  C(1+|\phi_k|,|\mu_k|)_\Omega + \eps^2 \|\nabla \phi_k\|_{L^2} \|\nabla\mu_k\|_{L^2}  \\
		&\leq \frac{1}{2}\norm{\mu_k}_{L^2}^2 + C(1+\norm{\phi_k}_{L^2}^2) + \frac{\eps^4}{m_0}\norm{\nabla \phi_k}_{{L}^2}^2+\frac{m_0}{8} \norm{\nabla\mu_k}_{{L}^2}^2 ,
	\end{align*}
	and consequently
	\begin{equation}
		\label{A_priori_source_terms_b_1}\norm{\mu_k}_{L^2}^2\leq C(1+\norm{\phi_k}_{L^2}^2 +\norm{\nabla \phi_k}_{{L}^2}^2) +\frac{m_0}{4} \norm{\nabla\mu_k}_{{L}^2}^2.
	\end{equation}
	It remains to estimate the first two integrals on the right-hand side of (\ref{energy_identity_3}). Using (\ref{divergence_equation_regularity}), Hölder's and Young's inequalities, we obtain 
	\begin{align}
		\nonumber (\nabla\phi_k\cdot{\u}_k +\phi_k\Gamma_{\vv,k},\mu_k)_\Omega &\leq (\norm{\nabla\phi_k}_{{L}^2}\norm{{\u}_k}_{{L}^{\infty}} + \norm{\phi_k}_{L^2}\norm{\Gamma_{\vv,k}}_{L^{\infty}})\norm{\mu_k}_{L^2}\\
		\nonumber & \leq C \norm{\Gamma_{\vv,k}}_{L^{\infty}}(\norm{\nabla\phi_k}_{{L}^2}+\norm{\phi_k}_{L^2})\norm{\mu_k}_{L^2}\\
		\label{A_priori_source_terms_b_5}& \leq C(\norm{\phi_k}^2_{L^2} + \norm{\nabla\phi_k}_{{L}^2}^2) + \delta\norm{\mu_k}^2_{L^2},
	\end{align}
	with $\delta>0$ chosen sufficiently small. Then $\|\mu_k\|_{L^2}^2$ can be estimated as in \cref{A_priori_source_terms_b_1} and $\|\phi_k\|_{L^2}^2$ is estimated using the growth condition of $\Psi$, see (A4).
	Thus, we obtain the following energy inequality
	\begin{equation} \label{energy_identity_4}\begin{aligned}
			&\ddt \left( \|\Psi(\phi_k)\|_{L^1}+\frac{\eps^2}{2}\|\nabla\phi_k\|_{L^2}^2 \right) +  \frac{\beta_0}{2} \|\vv_k\|_{L^{s+1}}^{s+1}+\frac{\alpha_0}{2}\|\vv_k\|^2_{L^2} + \frac{m_0}{2}\|\nabla\mu_k\|^2_{L^2} \\
			&\lesssim 1 +\norm{\nabla\phi_k}_{{L}^2}^2+ \norm{\Psi(\phi_k)}_{L^1}.
	\end{aligned}\end{equation}
	By a Gronwall-type argument and the definition of the initial data, we deduce the uniform energy estimate
	\begin{equation} \label{energy_identity_5}\begin{aligned}
			& \|\Psi(\phi_k)\|_{L^\infty(L^1)}+\|\nabla\phi_k\|_{L^\infty(L^2)}^2  +  \|\vv_k\|_{L^{s+1}(L^{s+1})}^{s+1}+\|\vv_k\|^2_{L^2(L^2)} + \|\nabla\mu_k\|^2_{L^2(L^2)} \lesssim 1 +\norm{\phi_0}_{H^1}^2.
	\end{aligned}\end{equation}
	
	\noindent\textbf{Step 3: 
		Higher-order estimates.}
	Using the Gagliardo--Nirenberg inequality, see \cite[Theorem 1.24]{roubicek}, and elliptic regularity theory, we deduce that
	\begin{equation*}
		\label{higher_order_estimates_2}
		\begin{aligned}\norm{\phi_k}_{L^{\infty}}&\lesssim \norm{\phi_k}_{H^1}^{\frac{1}{2}}\norm{\phi_k}_{H^2}^{\frac{1}{2}}  \lesssim \norm{\phi_k}_{H^1}^{\frac{1}{2}}\left(\norm{\phi_k}^{\frac{1}{2}}_{L^2}+\norm{\Delta\phi_k}^{\frac{1}{2}}_{L^2}\right).
		\end{aligned}
	\end{equation*}
	Choosing $v=\Delta \phi_k$ in \cref{Eq:GalerkinMu}, we obtain by Hölder's inequality and the assumption on $\Psi$, see (A4), that
	\begin{equation*}
		\label{higher_order_estimates_4}
		\begin{aligned}\eps\norm{\Delta\phi_k}^2_{L^2(L^2)} &=  (\nabla\mu_k,\nabla\phi_k)_{\O_T} - (\Psi''(\phi_k),|\nabla\phi_k|^2)_{\O_T} \\
			&\leq  \norm{\nabla\mu_k}_{{L}^2(L^2)}\norm{\nabla\phi_k}_{{L}^2(L^2)}  +  C (1+|\phi_k|^q,|\nabla\phi_k|^2)_{\O_T} \\ 
			&\lesssim 1+ (|\phi_k|^q,|\nabla\phi_k|^2)_\OT.
		\end{aligned}
	\end{equation*}
	In the case of $q=0$, it is clearly bounded.
	In the case of $q\in (0,4)$, we use Hölder's inequality to obtain
	\begin{equation}\label{higher_order_estimates_8}
		\begin{aligned}
			(|\phi_k|^q,|\nabla\phi_k|^2)_\OT &\lesssim  (\norm{\phi_k}_{L^{\infty}}^q,\norm{\nabla\phi_k}_{{L}^2}^2)_{(0,T)}\\
			&\lesssim (\norm{\phi_k}^2_{H^1}\norm{\phi_k}^{\frac{q}{2}}_{H^1},\norm{\phi_k}^{\frac{q}{2}}_{L^2} + \norm{\Delta\phi_k}_{L^2}^{\frac{q}{2}})_{(0,T)}\\
			&\lesssim \norm{\phi_k}_{L^{\infty}(H^1)}^{q+2} + (\norm{\phi_k}^{\frac{q+4}{2}}_{L^2},\norm{\Delta\phi_k}^{\frac{q}{2}}_{L^2})_{(0,T)}.
		\end{aligned}
	\end{equation}
	We note that $\frac{4}{q}>1$ and thus, we use Young's inequality to estimate the last integral on the right-hand side of (\ref{higher_order_estimates_8}) by
	\begin{equation*}
		\label{higher_order_estimates_9}(\norm{\phi_k}^{\frac{q+4}{2}}_{H^1},\norm{\Delta\phi_k}^{\frac{q}{2}}_{L^2})_{(0,T)} \leq C\norm{\phi_k}_{L^{\infty}(H^1)}^{\frac{2(q+4)}{4-q}} + \frac{\eps^2}{2}\norm{\Delta\phi_k}_{L^2(L^2)}^2.
	\end{equation*}
	Consequently, we obtain $\frac{\eps^2}{2}\norm{\Delta\phi_k}_{L^2(L^2)}^2\leq C$, which implies by elliptic regularity theory that $\phi_k$ is uniformly bounded in $L^2(0,T;H^2(\Omega))$.
	Furthermore, we choose $v=\Delta^2 \phi_k$ in \cref{Eq:GalerkinMu} to obtain similarly to before
	$$\begin{aligned}\eps \|\nabla\!\Delta \phi_k\|_{L^2(L^2)}^2 &= (\nabla \mu_k,\nabla\!\Delta \phi_k)_{\O_T} + (\Psi''(\phi_k)\nabla \phi_k , \nabla\!\Delta \phi_k)_{\O_T} \\
		&\leq \frac{\eps}{2} \|\nabla\!\Delta \phi_k\|_{L^2(L^2)}^2 + C\|\nabla\mu_k\|_{L^2(L^2)}^2 + C\left(1+\|\phi_k\|_{L^2(L^\infty)}^{2q}\right) \|\nabla \phi_k\|_{L^\infty(L^2)}^2.
	\end{aligned}$$
	Then, we deduce from the uniform bound of $\nabla\!\Delta\phi_k$ in $L^2(\O_T)$ and elliptic regularity theory that
	$$\|\phi_k\|_{L^2(H^3)} \leq C.$$
	Lastly, by the Lions--Peetre interpolation method, we obtain that $\phi_k$ is uniformly bounded in $L^4(0,T;H^2(\Omega))$ and likewise $\mu_k$ in $L^4(0,T;L^2(\Omega))$. \medskip
	
	\noindent\textbf{Step 4:
		Estimation of the pressure.}
	Taking the scalar product of \cref{Eq:GalerkinP} with an arbitrary function $\ww\in L^\sigma(0,T;L^{s+1}(\O)^d)$, $\sigma \geq 1$ to be deduced, integrating over $\Omega_T$, we obtain
	\begin{equation}\label{Eq:Pressure}
		-(\nabla p_k,\ww)_{\O_T} =  (\alpha_k \vv_k+\beta_k|\vv_k|^{s-1} \vv_k - \mu_k\nabla \phi_k ,\ww)_{\O_T} =:R_k(w).
	\end{equation}
	In the following, we distinguish the cases of $s \in (1,2)$ and $s \geq 2$. First, in the case of $s \geq 2$, we have with Hölder's inequality 
	that
	\begin{equation*}\begin{aligned}
			|R_k(\ww)| &\lesssim \int_0^T \left( \norm{\vv_k}_{L^2} \|\ww\|_{L^2} + \|\vv_k\|_{L^{s+1}}^s \|\ww\|_{L^{s+1}}+\norm{\mu_k}_{L^6}\norm{\nabla\phi_k}_{{L}^2} \|\ww\|_{L^3} \right) \dt \\
			&\lesssim \norm{\vv_k}_{L^2(L^{2})} \|\ww\|_{L^2(L^2)} + \|\vv_k\|_{L^{s+1}(L^{s+1})}^s \|\ww\|_{L^{s+1}(L^{s+1})}  +\norm{\mu_k}_{L^2(H^1)}\norm{\nabla\phi_k}_{L^\infty({L}^2)} \|\ww\|_{L^{2}(L^3)}
			\\&\lesssim \|\ww\|_{L^{s+1}(L^{s+1})}.
		\end{aligned}
	\end{equation*} 
	Furthermore, for $s \in (1,2)$ the estimates are more delicate as we want to maintain the spatial $L^{s+1}(\Omega)$ regularity of $\ww$ but thus we cannot estimate $\phi_k$ using the $L^\infty(0,T;H^1(\Omega))$-estimate. Instead, we use the Gagliardo--Nirenberg inequality, see \cite[Theorem 1.24]{roubicek}, to deduce for $s \in (1,2)$ that
	$$\phi_k \in L^\infty(0,T;H^1(\Omega)) \cap L^2(0,T;H^3(\Omega)) \hookrightarrow L^{4(s + 1)/(2 - s)}(0,T;W^{1,6(s+1)/(5s-1)}(\Omega)),$$
	and thus
	\begin{equation*}\begin{aligned}
			|R_k(\ww)| &\lesssim \|\ww\|_{L^{s+1}(L^{s+1})}+\int_0^T  \norm{\mu_k}_{L^6}\norm{\nabla\phi_k}_{{L}^{(6s+6)/(5s-1)}} \|\ww\|_{L^{s+1}}  \dt \\
			&\lesssim   \|\ww\|_{L^{s+1}(L^{s+1})}+\norm{\mu_k}_{L^2(H^1)}\norm{\nabla\phi_k}_{L^{4(s + 1)/(2 - s)}(L^{6(s+1)/(5s-1)})} \|\ww\|_{L^{4(s+1)/(3s)}(L^{s+1})} \\
			&\lesssim   \|\ww\|_{L^{\max\{s+1,4(s+1)/(3s)\}}(L^{s+1})}.
		\end{aligned}
	\end{equation*} 
	Hence, $R_k$ is bounded in the dual space of $L^\sigma(0,T;L^{s+1}(\Omega)^d)$ where we have defined $\sigma$ by
	$$\sigma = \begin{cases}
		s+1, &s \geq 4/3, \\
		\tfrac43 (1+1/s), &s \in (1,4/3).
	\end{cases}$$
	%s and 4/3
	Together, taking the supremum over all ${\ww}\in L^\sigma(0,T;L^{s+1}(\O)^d)$, we deduce that $p_k$ is bounded in the Bochner space $L^{\sigma'}(0,T;W_0^{1+1/s}(\Omega))$. We note that it holds $\sigma' \in (1,\frac74]$ and $\sigma\geq \frac73$, with the specific value depending on $s$. \medskip
	
	\noindent\textbf{Step 5: 
		Regularity for the convection terms.} To pass to the limit $k \to \infty$ later on, we investigate the boundedness of the term $\div(\phi_kv_k)$, which appears in \cref{Eq:GalerkinPhi}.
	Again, we use the Gagliardo--Nirenberg interpolation inequality to deduce the existence of some
	$\kappa \in (8,\infty)$ such that
	$$L^2(0,T;H^3(\Omega)) \cap L^\infty(0,T;H^1(\Omega)) \hookrightarrow L^\kappa(0,T;L^{6\kappa/(\kappa-8)}(\Omega)) \cap L^{8}(\Omega_T).$$ 
	First, let $s>7$. Then we choose $\kappa=s+1$ to obtain that $\phi_k$ is uniformly bounded in
	$$L^{s+1}(0,T;L^{6(s+1)/(s-7)}(\Omega)) \hookrightarrow L^{s+1}(0,T;L^{2(s+1)/(s+3)}(\Omega)), \quad s>8.$$
	For $s \in (1,7]$, the bound is valid in the same function space since $\phi_k$ is bounded in $$L^8(\O_T) \hookrightarrow L^{s+1}(0,T;L^{2(s+1)/(s+3)}(\Omega)), \quad s \in (1,7].$$
	Together, we deduce for the convection term that
	\begin{equation}
		\label{Apriori_convection_terms_1}
		\begin{aligned}
			\norm{\div(\phi_k\vv_k)}_{L^{s+1}(L^{2(s+1)/(s+3)})}  
			&\leq \norm{\nabla\phi_k  \cdot \vv_k}_{L^{s+1}(L^{2(s+1)/(s+3)})}+\norm{\phi_k \Gamma_{\vv,k}}_{L^{s+1}(L^{2(s+1)/(s+3)})} \\
			&\leq\norm{\phi_k}_{L^{\infty}(H^1)}\norm{\vv_k}_{L^{s+1}(L^{s+1})} + C\norm{\phi_k}_{L^{s+1}(L^{2(s+1)/(s+3)})} \\
			&\leq C.
		\end{aligned}
	\end{equation}
	
	\noindent\textbf{Step 6: 
		Regularity for the time derivative.}
	We consider an arbitrary function $$\zeta\in L^2\left(0,T;H^1(\O)\cap L^{2(s+1)/(s-1)}(\O)\right),$$ and coefficients $\{\zeta_{kj}\}_{1\leq j\leq k}$ such that $\mathbb{P}_k\zeta = \sum_{j=1}^{k}\zeta_{kj}w_j$. Then taking the test function $\mathbb{P}_k\zeta$ in \cref{Eq:GalerkinPhi} and using the uniform energy estimates, it gives 
	\begin{equation}
		\begin{aligned}
			\langle \pt \phi_k,\mathbb{P}_k\zeta \rangle &\lesssim \|\xi\|_{L^2(H^1)} + |(\div(\phi_k \vv_k),\xi)_\OT|
			\\&\lesssim \|\xi\|_{L^2(H^1)} + \norm{\div(\phi_k\vv_k)}_{L^{s+1}(L^{2(s+1)/(s+3)})} \|\xi\|_{L^{1+1/s}(L^{2(s+1)/(s-1)})} \\
			&\lesssim \|\xi\|_{L^2(H^1)} +  \|\xi\|_{L^2(L^{2(s+1)/(s-1)})}
		\end{aligned}
	\end{equation}
	Thus, we obtain
	\begin{equation}
		\label{Apriori_time_derivatives_2}\norm{\p_t\phi_k}_{L^2((H^1 \cap L^{2(s+1)/(s-1)})')} \leq C.
	\end{equation}
	
	\noindent\textbf{Step 7: 
		Limit passage.} 
	We first state the time-integrated system where we want to pass to the limit $k \to \infty$. In fact, for every $\eta \in C_c^1([0,T])$, $\u \in L^{s+1}(\Omega)^d$, $q \in L^{1+1/s}(\Omega)$ and $1 \leq j \leq k$, we have the discrete system 
	\begin{subequations}
		\begin{align}
			\label{Eq:GalerkinPhiTime}
			-(\phi_k,\xi_j\eta')_\OT  +(m(\phi_k)\nabla\mu_k,\nabla \xi_j\eta)_\OT -(\Gamma_{\phi,k},\xi_j\eta)_\OT   +(\nabla\phi_k \cdot \vv_k,\xi_j\eta)_\OT-(\phi_k\Gamma_{\vv,k},\xi_j\eta)_\OT &=0,\\
			\label{Eq:GalerkinMuTime}
			(\mu_k,\xi_j\eta)_\OT - \eps (\nabla \phi_k,\nabla \xi_j \eta)_\OT - (\Psi'(\phi_k),\xi_j\eta )_\OT &=0, \\
			\label{Eq:GalerkinForchTime} (\nabla p_k,\u\eta)_\OT + (\beta_k|\vv_k|^{s-1}\vv_k,\u\eta)_\OT +(\alpha_k \vv_k,\u\eta)_\OT-(\mu_k\nabla \phi_k,\u\eta)_\OT &= 0,\\
			\label{Eq:GalerkinDivTime} (\div \vv_k,q\eta)_\OT - (\Gamma_{\vv,k},q\eta)_\OT &=0.
		\end{align}
	\end{subequations}
	We collect the previously derived uniform estimates:
	\begin{equation}\label{central_estimate}\begin{aligned}
			&\norm{\phi_k}_{L^{\infty}(H^1)\cap L^2(H^2)\cap H^1((H^1\cap L^{2(s+1)/(s-1)})')} + \norm{\mu_k}_{L^2(H^1)}   + \norm{\div(\phi_k\vv_k)}_{L^{s+1}(L^{2(s+1)/(s+3)})} \\&+ \norm{\vv_k}_{L^{s+1}(L^{s+1}_\div)} + \norm{p_k}_{L^{\sigma'}(W_0^{1+1/s})} \leq C.
	\end{aligned}\end{equation}
	We obtain the existence of limit functions $(\phi,\mu,p,\vv)$ with the subsequential (with the same index) convergences:
	\begin{alignat*}{3}
		\phi_k&\rightharpoonup\phi&&\quad\text{weakly-}*&&\quad\text{in }L^{\infty}(0,T;H^1(\Omega))\cap L^2(0,T;H^3(\Omega))  \cap H^1(0,T;(H^1(\Omega)\cap L^{2(s+1)/(s-1)}(\Omega))'),\\
		\mu_k&\rightharpoonup \mu&&\quad\text{weakly}&&\quad\text{in }L^2(0,T;H^1(\Omega)),\\
		p_k&\rightharpoonup p&&\quad\text{weakly}&&\quad\text{in } L^{\sigma'}(0,T;W_0^{1+1/s}(\Omega)),\\
		\vv_k&\rightharpoonup \vv&&\quad\text{weakly}&&\quad\text{in }L^{s+1}(0,T;L_\div^{s+1}(\Omega)),\\
		\div(\phi_k\vv_k)&\rightharpoonup \tau&&\quad\text{weakly}&&\quad\text{in }{L^{s+1}(0,T;L^{2(s+1)/(s+3)}(\Omega))},
	\end{alignat*}
	for some limit function $\tau \in {L^{s+1}(0,T;L^{2(s+1)/(s+3)}(\Omega))}$.
	We apply the Aubin--Lions lemma with the Gelfand triples
	$$\begin{aligned} H^3(\Omega) \com W^{2,r}(\Omega)  &\con (H^1(\Omega)\cap L^{2(s+1)/(s-1)}(\Omega))', &&\quad \forall r \in [1,6), \\
		H^1(\Omega) \com L^r(\Omega)  &\con (H^1(\Omega)\cap L^{2(s+1)/(s-1)}(\Omega))', &&\quad \forall r \in [1,6), \end{aligned}$$
	to conclude the strong convergence:
	\begin{alignat*}{3}
		\phi_k&\rightarrow\phi&&\quad\text{strongly}&&\quad\text{in }C^0([0,T];L^r(\Omega))\cap L^2(0,T;W^{2,r}(\Omega)), \quad \forall r\in[1,6).
	\end{alignat*}
	By the uniqueness of limits, we directly obtain $\phi(0)=\phi_0$ in $L^r(\Omega)$ for any $r \in [1,6)$.
	
	The convergence of all the linear terms can be directly treated using weak convergences. The convergence of the term $(\Gamma_{\phi,k},\xi_j\eta)_\OT$ can be treated using the strong convergence of $\phi_k$, the boundedness of $\Gamma_\phi$ and the Lebesgue dominated convergence theorem. The product terms $(m(\phi_k)\nabla\mu_k,\nabla\xi_j\eta)_\OT$, $(\phi_k \Gamma_{\vv,k},\xi_j \eta)_\OT$, $(\alpha_k\vv_k,q\eta)_\OT$, $(\mu_k\nabla\phi_k)$ follow due to the weak-strong convergence lemma because there is always a product of a strongly and a weakly converging sequence. By passing limits in \cref{Eq:GalerkinPhiTime} and integrating by parts back and forth, we find that the limit function $\tau$ is indeed equal to $\div(\phi v)$.
	
	It remains to treat the nonlinear operator $F:L^{s+1}(\Omega_T)^d\to L^{1+1/s}(\Omega_T)^d$ with $F(\vv_k)=|\vv_k|^{s-1} \vv_k$. We note that $F$ is monotone and hemicontinuous, see \cref{Appendix}, and bounded by \cref{central_estimate}, which implies by Minty's trick, see \cite[Lemma 2.1]{andreianov2017nonlinear}, that
	$$F(\vv_k) \rightharpoonup F(\vv) \text{ weakly in } L^{1+1/s}(\Omega_T)^d.$$
	Together with the strong convergence of $\phi_k$ and the boundedness of $\beta$, we get that that $\beta(\phi_k)$ converges strongly a.e.~in $\O_T$ and  thus, $\beta(\phi_k)F(\vv_k)$ converges weakly in $L^{1+1/s}(\Omega_T)^d$. In fact, we obtain
	$$(\beta(\phi_k)F(\vv_k),\u\eta )_{\Omega_T} \to (\beta(\phi)F(\vv),\u\eta )_{\Omega_T}.$$
	
	\noindent\textbf{Step 8: Limit energy inequality}
	Finally, we consider the limit in the energy identity. 
	We first recall \cref{Eq:Test1} which reads
	\begin{equation}\label{Eq:Testx}
		\begin{aligned}
			&\ddt \left( \|\Psi(\phi_k)\|_{L^1}+\frac{\eps^2}{2}\|\nabla\phi_k\|^2_{L^2} \right)  + (m(\phi_k),|\nabla\mu_k|^2)_\Omega =  (\Gamma_{\phi,k},\mu_k)_\Omega -  (\nabla \phi_k\cdot \vv_k + \phi_k \Gamma_{\vv,k},\mu_k)_\Omega. 
	\end{aligned}\end{equation}
	Multiplying \cref{Eq:GalerkinP} with $\vv_k$, integrating over $\Omega$ and by parts, we end up at
	$$\begin{aligned}
		-(p_k,\div(\vv_k))_\Omega+(\beta_k|v_k|^{s-1} \vv_k,\vv_k)_\Omega +  (\alpha_k \vv_k,\vv_k)_\Omega= (\mu_k\nabla\phi_k,\vv_k)_\Omega.
	\end{aligned}$$
	Summation and integration over  $(0,T)$ yields the discrete energy identity
	\begin{equation}\label{Eq:Testz}
		\begin{aligned}
			\mathcal{E}(\phi_k)(t) + \int_0^t\mathcal{D}(\mu_k,\vv_k) = \mathcal{E}(\phi_k)(0) + \int_0^t(\mu_k,\Gamma_{\phi,k})_\Omega + (p_k,\Gamma_{\vv,k})_\Omega -  (\phi_k\mu_k,\Gamma_{\vv,k})_\Omega
	\end{aligned}\end{equation}
	with dissipation
	\begin{equation*}
		\mathcal{D}(\mu_k,\vv_k):= (m(\phi_k),|\nabla\mu_k|^2)_\Omega +(\beta_k|v_k|^{s-1} \vv_k,\vv_k)_\Omega +  (\alpha_k \vv_k,\vv_k)_\Omega.
	\end{equation*}
	Note that we exchanged all $\div(\vv_k)$ with $\Gamma_{\vv,k}$, which is almost everywhere the same.
	When passing to the limit in \cref{Eq:Testz}, the energy and the dissipation are treated by the weak lower semi-continuity of norms and Fatou's lemma for $\Psi$. The initial energy is strongly converging. Hence, we have to consider the production terms on the right-hand side. Due to strong convergence of $\phi_k$ and continuity of $\Gamma_\phi,\Gamma_\vv$, we obtain strong and almost everywhere convergence of $\Gamma_{\phi,k}$ and $\Gamma_{\vv,k}.$ We may note that $\phi_k\Gamma_{\vv,k}$ strongly converges in $L^2(\Omega_T).$ Hence, we can pass to the limit in all terms using the weak-strong convergence lemma, which yields
	\begin{equation}\label{Eq:Energy}
		\begin{aligned}
			\mathcal{E}(\phi)(t) + \int_0^t\mathcal{D}(\mu,\vv) \leq \mathcal{E}(\phi)(0) + \int_0^t(\mu,\Gamma_{\phi})_\Omega + (p,\Gamma_{\vv})_\Omega -  (\phi\mu,\Gamma_{\vv})_\Omega.
	\end{aligned}\end{equation}
\end{proof}

\section{Structure-preserving approximation} \label{Sec:Appro}
As a preparatory step, we introduce the relevant notation and assumptions regarding the discretisation strategy for the Cahn--Hilliard--Forchheimer system. For spatial discretization, we require that $\Th$ is a geometrically conforming partition of $\Omega \subset \R^d $ into simplices. For simplicity, we present all estimates for $d=3$. The lower dimensions can be obtained similarly, if not much simpler. 

First, to derive a suitable variational scheme for discretization, we use integration by parts as usual. The only term where we have to be careful is the Kortweg stress $\nabla\phi\mu$ for which we use the following identity
\begin{align*}
	\la \nabla p -\phi\nabla\mu,\ww\ra_\Omega &= \la \nabla p - \nabla(\phi\mu) + \phi\nabla\mu,\ww\ra_\Omega \\
	&= -\la p - \phi\mu, \div(\ww)\ra_\Omega + \la p-\phi\mu,\ww\cdot\nvec \ra_{\partial\Omega} + \la \phi\nabla\mu,\ww \ra_\Omega \\
	&= -\la \tilde p, \div(\ww)\ra_\Omega - \la \phi\mu,\ww\cdot\nvec \ra_{\partial\Omega} + \la \phi\nabla\mu,\ww \ra_\Omega.
\end{align*}
In the following, we use a typical abuse notation and always simply write $p$ even though we actually use $p$ if $\gamma=1$ and $\tilde p$ if $\gamma=0.$  We denote the space of continuous, piecewise linear functions over $\Th$ by $\Vh$, the Raviart--Thomas space of order one by $\Xh$, as well as the space of (discontinuous) piece-wise linear over $\Th$ by $\Qh$, that is, we define
\begin{align*}
	\Vh &:= \{v \in H^1(\Omega)\cap C^0(\bar\Omega) : v|_K \in P_1(K) \quad \forall K \in \Th\}, \\ 
	\Xh &:= \{\vv \in L^2_{\div}(\Omega): \vv|_K \in \mathrm{RT}_1(K)=P_1(K)^d\oplus \mathbf{x}P_1(K)\quad \forall K \in \Th \}, \\
	\Qh &:= \{v \in L^2(\Omega) : v|_K \in P_1(K) \quad \forall K \in \Th\}.
\end{align*}
Here, $P_1^H(K)$ are the homogeneous polynomials of degree 1 on an element $K.$ Further, we define the discrete Laplacian $\Delta_h:\Vh\to\{v\in\Vh:\la v,1 \ra_\Omega=0\}$ and the Raviart--Thomas projection $\Pi_h:X\to \Xh$ by
\begin{equation}\label{eq:RTproj}
	\begin{aligned}
		\la \Delta_hv_h,\psi_h \ra_\Omega &= - \la \nabla v_h,\nabla \psi_h \ra_\Omega &&\quad\forall \psi_h\in\Vh, \\
		\la \div(\Pi_h\vv-\vv),q_h \ra_\Omega &= 0 &&\quad\forall q_h\in \Qh,
	\end{aligned}
\end{equation}
and the following errors estimates hold, cf.~\cite{pan2012mixed}:
\begin{align}
	\norm{\Pi_h\vv-\vv}_{L^q} &\leq Ch^s\norm{\vv}_{W^{s,q}} &&\hspace{-2.3cm}\forall \vv\in X\cap W^{s,q}(\Omega)^d \qquad\,\text{ for } \frac{1}{q}<s\leq 2, \label{eq:RTproj_Lq}\\
	\norm{\div(\Pi_h\vv-\vv)}_{L^2} &\leq Ch^s\norm{\div(\vv)}_{H^{s}} &&\hspace{-2.3cm} \forall \vv\in X\cap H^s(\div,\Omega) \qquad\text{for } 0 \leq s\leq 2. \label{eq:RTproj_Div}
\end{align}
In the upcoming numerical analysis of the fully discrete scheme, we require the following discrete interpolation inequalities.
\begin{lemma}[Discrete interpolation inequalities\cite{Diegel2017,Heywood1982}]
	Let $\Omega$ a convex polyhedral domain, and $\Th$ is assumed to be a globally quasi-uniform triangulation of $\Omega$. Then the following holds
	\begin{align}
		\norm{\psi_h}_{L^\infty} +  \norm{\nabla\psi_h}_{L^3} \lesssim \norm{\Delta_h\psi_h}_{L^2}^{1/2}\norm{\psi_h}_{L^6}^{1/2} + \norm{\psi_h}_{L^6} \qquad \forall \psi_h \in \Vh.    \label{eq:discreteinftyinterp} 
	\end{align}
\end{lemma}
Furthermore, to handle the Forchheimer equation, we make use of an inf-sup stability result.
\begin{lemma}[Inf-sup stability\cite{pan2012mixed}]
	There exists a positive constant $c$ independent of $h$ such that  
	\begin{align}
		\inf_{q_h\in\Qh}\sup_{\vv_h\in\Xh} \frac{\la \div(\vv_h),q_h \ra_\Omega}{\norm{q_h}_{\Qh}\norm{\vv_h}_{\Xh}} \geq c. \label{eq:infsup}
	\end{align}
\end{lemma}
In order to obtain analytical result for the numerical method, we reinforce (A4) in \cref{Assumptions}  as follows:

\begin{assumption}\label{Assumption2} Let \cref{Assumptions} hold and additionally:
	\begin{itemize}[leftmargin=1.2cm] \itemsep.1em
		\item[(A5)] It holds (A4) and $\Psi\in C^3(\R,\R_{\geq 0})$ with
		$|\Psi''(s)|\lesssim 1+|s|^q$, $|\Psi'''(s)|\lesssim 1+|s|^{q-1}$ for some $q \in [0,2)$ such that and  $|\Psi''(s)|> - C$. 
		\item[(A6)] We assume that $\Omega$ is a convex polyhedral domain, and $\Th$ is assumed to be a globally quasi-uniform triangulation of $\Omega$
	\end{itemize}    
\end{assumption}

Regarding the time discretization, we introduce the following notions.
We divide the time interval $[0,T]$ into uniform steps with step size $\tau>0$ and introduce $\Itau:=\{0=t^0,t^1=\tau,\ldots, t^{n_T}=T\}$, where $n_T=\tfrac{T}{\tau}$ is the absolute number of time steps. We denote by $\Pi^1_c(\Itau;X)$ and $\Pi^0(\Itau;X)$ the spaces of continuous piecewise linear and piecewise constant functions on $\Itau$ with values in the space $X$, respectively. By $g^{n+1}$ and $g^n$ we denote the evaluation/approximation of a function $g$ in $\Pi^1_c(\Itau)$ or $\Pi^0(\Itau)$ at $t=\{t^{n+1},t^n\}$, respectively, and write $I_n=(t^n,t^{n+1})$. The discrete-time derivative is denoted by $d^{n+1}_\tau g := \frac{g^{n+1}-g^n}{\tau}$.
The time average of a function $g(\rho)$ for $\rho\in\Pi^1_c(\Itau)$ is given by
$g_\av(\rho) := \frac{1}{\tau}\int_{I_n} g(\rho(s)) \ds.$

\begin{problem}\label{prob:ac2}
	Let the initial data $\phi_{h}^0\in \Vh$ be given. Find $\phi_h\in \Pi^1_c(\Itau;\Vh)$ and $(\mu_{h},\vv_h,p_h)\in \Pi^0(\Itau;\Vh\times\Xh\times\Qh)$ that satisfy the variational system
	\begin{align}
		&\la\dtau\phi_h,\psi_h\ra_\Omega  - \la \phi_h^n\vv_h^{n+1},\nabla\psi_h \ra_\Omega + \gamma\la \phi_h^n\vv_h^{n+1}\cdot\nvec,\psi_h \ra_{\partial\Omega} + \la m(\phi_h^n)\nabla\mu_h^{n+1},\nabla\psi_h \ra_\Omega - \la \Gamma_\phi(\phi_h^{n}),\psi_h \ra_\Omega = 0, \label{eq:pg1} \\[.2cm]
		&\la \mu_h^{n+1}, \xi_h\ra_\Omega - \varepsilon^2\la \nabla\phi_h^{n+1},\nabla\xi_h \ra_\Omega -  \la \Psi'_\av(\phi_h),\xi_h \ra_\Omega = 0, \label{eq:pg2}\\[.2cm]
		&\left( (\alpha(\phi_h^n) + \beta(\phi_h^n)|\vv_h^{n+1}|^{s-1})\vv_h^{n+1},\ww_h \right)_\Omega - \la p_h^{n+1},\div(\ww_h) \ra_\Omega + \la \phi^n_h\nabla\mu_h^{n+1},\ww_h \ra_\Omega + \gamma\la \phi_h^n\mu_h^{n+1},\div(\ww_h) \ra_\Omega \notag \\%+\la p_D,\ww_h\cdot\nvec \ra_{\partial\Omega} \\ 
		& - \gamma\la \phi_h^n\mu_h^{n+1},\ww_h\cdot\nvec \ra_{\partial\Omega}=0, \label{eq:pg3}\\[.2cm]
		&\la \div(\vv_h^{n+1}),q_h \ra_\Omega - \la \Gamma_\vv(\phi_h^n),q_h \ra_\Omega = 0 \label{eq:pg4}
	\end{align}
	for every $(\psi_h,\xi_h,\ww_h,q_h)\in\Vh\times\Vh\times\Xh\times\Qh$ and every $0\leq n\leq n_T-1$.
\end{problem}
We show the well-posedness of the fully discrete scheme in \cref{prob:ac2} and further prove that the total mass and energy-dissipation are preserved. Here, the discrete dissipation and source rate are defined by
\begin{align*}
	\int_{t^m}^{t^n}\mathcal{D}(\mu_h,\vv_h) \ds &:= \tau\sum_{k=m}^{n-1} \la m(\phi_h^k)\nabla\mu_h^{k+1},\nabla\mu_h^{k+1} \ra_\Omega + \la (\alpha(\phi_h^k)+\beta(\phi_h^k)\snorm{\vv_h^{k+1}}^{s-1})\vv_h^{k+1},\vv_h^{k+1} \ra_\Omega, \\
	\int_{t^m}^{t^n} \mathcal{P}(\mu_h,p_h) \ds &:= \tau\sum_{k=m}^{n-1} \la \Gamma_\phi(\phi_h^{k}),\mu_h^{k+1} \ra_\Omega +  \la \Gamma_\vv(\phi_h^k),p_h^{k+1} \ra_\Omega - \gamma\la \phi_h^k\mu_h^{k+1},\div(\vv_h^{k+1}) \ra_\Omega.
\end{align*}
Then our main result on the structure-preserving analysis of  \cref{prob:ac2} reads as follows:
\begin{theorem}
	Let \cref{Assumption2} hold. Furthermore, let $h,\tau>0$ be given such that $\tau\leq C_e$ where $C_e>0$ depends solely on the parameters and external forces. Then, for any $\phi_{h}^0 \in \Vh$, \cref{prob:ac2} admits at least one solution $(\phi_h,\mu_h,\vv_h,p_h)$. Moreover, any such solution conserves the total mass and energy-dissipation balance, that is, it holds
	\begin{align}
		\la \phi_h^{n},1 \ra_\Omega = \la \phi_h^m,1 \ra_\Omega &+ \tau\sum_{k=m}^{n-1}\la \Gamma_\phi(\phi_h^k),1 \ra_\Omega - \gamma\la \phi_h^k,\vv_h^{k+1}\cdot\nvec \ra_{\partial\Omega}, \label{thm1:mass:phi}  \\
		\mathcal{E}(\phi_h)\vert_{t^m}^{t^n} &+\int_{t^m}^{t^n}\mathcal{D}(\mu_h,\vv_h) \ds\leq  \int_{t^m}^{t^n} \mathcal{P}(\mu_h,p_h) \ds. \label{thm1:mass:energy}
	\end{align}
	Furthermore, if the time step restriction $\tau\leq C_u$ holds, then the solution is unique. Here, the constant $C_u>0$ depends solely on the parameters and the initial data.
\end{theorem}

\begin{remark} We make the following remarks:
	\begin{itemize}
		\item The explicit evaluation at $\phi_h^n$ in the parameter functions and the transport term is a consequence of the uniqueness proof. The balance of mass and energy-dissipation also holds for implicit evaluation.
		\item The same results can be achieved almost verbatim if the parametric functions $\alpha,\beta,m,\Gamma_\phi,\Gamma_\vv$ are replaced by suitable projections or interpolations as long as they preserve a positive lower bound of $\alpha,\beta,m$. Similarly, one could also work with Brezzi--Douglas--Marini (BDM) spaces instead of Raviart--Thomas spaces.
		\item In principle, the estimates derived in the upcoming existence proof are uniform in $h,\tau$. Upon deriving some further estimates for $\nabla\Delta_h\phi$, the scheme can also be used as an approximate sequence to show the existence of weak solutions.
	\end{itemize}
	
\end{remark}

\begin{proof} First, we prove the structure-preserving properties of the scheme in \cref{prob:ac2}. Afterward, we derive suitable a-priori error bounds to show the existence and uniqueness of discrete solutions. \medskip
	
	\noindent\textbf{Step 1: Mass and energy-dissipation balance.}
	Regarding the balance of total mass, see \eqref{thm1:mass:phi}, we insert the test function $\psi_h=1$ in \eqref{eq:pg1} to deduce that
	\begin{align*}
		\restrb{\la \phi_h, 1\ra}{t^n}{t^{n+1}} = \int_{I_n}\la\pt\phi_h, 1\ra_\Omega &= \int_{I_n} \la \dtau\phi_h,1 \ra_\Omega \\
		&= - \int_{I_n}\la m(\phi_h^n)\nabla\mu_h^{n+1},\nabla 1 \ra_\Omega - \la \phi_h^n\vv_h^{n+1},\nabla 1 \ra_\Omega + \la \phi_h^n,\vv_h^{n+1}\cdot\nvec \ra_{\partial\Omega}+ \la \Gamma_\phi(\phi_h^{n}),1 \ra_\Omega \\
		&= \tau\left(\la \Gamma_\phi(\phi_h^{n}),1 \ra_\Omega - \la \phi_h^n,\vv_h^{n+1}\cdot\nvec \ra_{\partial\Omega}\right).
	\end{align*}
	The sum of the relevant time steps yields the result \eqref{thm1:mass:phi}. Let us now prove the energy-dissipation balance as stated in \eqref{thm1:mass:energy}. We have
	\begin{align*}
		\frac{1}{\tau}\restrb{\mathcal{E}(\phi_h)}{t^n}{t^{n+1}} & = \frac{1}{\tau}_{I_n} \varepsilon^2\la \nabla\pt\phi_h,\nabla\phi_h \ra_\Omega + \la \pt\phi_h,\Psi'(\phi_h) \ra_\Omega \\
		& = \varepsilon^2\la \nabla\phi^{n+1},\dtau\nabla\phi_h \ra_\Omega - \frac{\tau\varepsilon^2}{2}\norm{\nabla\dtau\phi_h}_{L^2}^2 + \la \Psi'_\av(\phi_h),\dtau\phi_h \ra_\Omega \\
		& \leq \varepsilon^2\la \nabla\phi^{n+1},\dtau\nabla\phi_h \ra_\Omega  + \la \Psi'_\av(\phi_h),\dtau\phi_h \ra_\Omega \\
		&= (a) + (b) , 
	\end{align*}
	where we introduced the short notation $(a)$ and $(b)$ for the two terms on the right-hand side. Inserting $\xi_h=\dtau\phi_h$ in \eqref{eq:pg2} and afterwards $\psi_h=\mu_h^{n+1}$ in \eqref{eq:pg1} yields
	\begin{align*}
		(a)  + (b)    =\,&  \la \mu_h^{n+1},\dtau\phi_h \ra_\Omega \\
		=\,&  - \la m(\phi_h^n)\nabla\mu_h^{n+1},\nabla\mu_h^{n+1} \ra_\Omega + \la \phi_h^n\vv_h^{n+1},\nabla\mu_h^{n+1}\ra_\Omega - \gamma\la \phi_h^n\vv_h^{n+1}\cdot\nvec,\mu_h^{n+1}\ra_{\partial\Omega} + \la \Gamma_\phi(\phi_h^n),\mu_h^{n+1} \ra_\Omega.
	\end{align*}
	Further we insert $\ww_h=\vv_h^{n+1}$ in \eqref{eq:pg3} and $q_h=p_h^{n+1}$ in \eqref{eq:pg4}, which implies:
	\begin{align*}
		0  =\,&  - \la m(\phi_h^n)\nabla\mu_h^{n+1},\nabla\mu_h^{n+1} \ra_\Omega - \la (\alpha(\phi_h^n)+\beta(\phi_h^n)\snorm{\vv_h^{n+1}}^{s-1})\vv_h^{n+1},\vv_h^{n+1}\ra_\Omega \\
		&+ \la \Gamma_\vv(\phi_h^n),p_h^{n+1}\ra_\Omega + \la \Gamma_\phi(\phi_h^n),\mu_h^{n+1} \ra_\Omega - \gamma\la \phi_h^n\mu_h^{n+1}, \div(\vv_h^{n+1})\ra_\Omega,    
	\end{align*}
	and summing over the relevant time steps yields the desired result \eqref{thm1:mass:energy}. \medskip
	
	\noindent\textbf{Step 2: a-priori bounds.}
	To derive suitable a-priori bounds, we mimic the proof in the Galerkin setting in \cref{Sec:Analysis}. First we consider the auxiliary problem:
	$$\begin{aligned}
		\div(\u^{n+1})&= \Gamma_\vv(\phi_h^n) &&\text{ in } \Omega, \\
		\u^{n+1} &= \frac{1}{|\partial\Omega|}\la \Gamma_\vv(\phi_h^n),1 \ra \nvec&&\text{ on } \p\Omega.
	\end{aligned}$$
	Since the right-hand side is bounded, see (A3) in \cref{Assumptions}, we apply \cref{Lem:Divergence} to establish the existence of a function $\u^{n+1}$ with the uniform bound
	\begin{align*}
		\norm{\u^{n+1}}_{W^{1,q}} \leq C\norm{\Gamma_\vv(\phi_h^n)}_{L^q} \qquad \forall q\in(1,\infty).
	\end{align*}   
	Using the Raviart--Thomas projection \eqref{eq:RTproj}, we then define $\u^{n+1}_h:=\Pi_h\u^{n+1}\in\Xh$ such that
	\begin{align}
		\la \div(\vv_h^{n+1}),p_h^{n+1} \ra &= \la \div(\u_h^{n+1}),p_h^{n+1} \ra, \label{eq:uhn+1} \\
		\norm{\div(\u^{n+1}_h)}_{L^2}^2 &\leq \norm{\div(\u^{n+1})}_{L^2}^2 \leq \norm{\Gamma_\vv(\phi_h^n)}_{L^2}^2 \leq C, \label{eq:divbounddisc}\\
		\norm{\u_h^{n+1}}_{L^q} &\leq  \norm{\u^{n+1}}_{L^q} +  \norm{\u^{n+1}-\u_h^{n+1}}_{L^q} \leq C, \label{eq:Lqbounddisc}
	\end{align}
	for every $p\in[1,\infty].$ We insert $q_h=p_h^{n+1}$ in \eqref{eq:pg3} and use \eqref{eq:uhn+1} to find
	\begin{align*}
		\la \Gamma_\vv(\phi_h^n),p_h^{n+1}\ra_\Omega = \la \div(\vv_h^{n+1}),p_h^{n+1} \ra_\Omega = \la \div(\u_h^{n+1}),p_h^{n+1} \ra_\Omega. 
	\end{align*}
	Moreover, choosing the test function $\ww_h=\u_h^{n+1}$ in \eqref{eq:pg3} yields
	\begin{align}
		\la \Gamma_\vv(\phi_h^n),p_h^{n+1}\ra_\Omega =\,&    \la (\alpha(\phi_h^n)+\beta(\phi_h^n)\snorm{\vv_h^{n+1}}^{s-1})\vv_h^{n+1}, \u_h^{n+1}\ra_\Omega - \la \phi_h^n\nabla\mu_h^{n+1},\u_h^{n+1} \ra_\Omega \notag\\
		&- \gamma\la \phi_h^n\mu_h^{n+1},\div(\u_h^{n+1}) \ra_\Omega + \gamma\la \phi_h^n\mu_h^{n+1},\u_h^{n+1}\cdot\nvec \ra_{\partial\Omega}. \label{eq:div_expansion}
	\end{align}
	Then, substracting \eqref{eq:pg4} and \eqref{eq:RTproj}, and using $q_h=\div(\u_h^{n+1}-\vv_h^{n+1})$ yields 
	\begin{equation*}
		\norm{ \div(\u_h^{n+1}-\vv_h^{n+1})}_{L^2}^2=0 ~\Longrightarrow~  \div(\u_h^{n+1}-\vv_h^{n+1})=0 \text{ point-wise in } \Omega.  
	\end{equation*}
	From the energy-dissipation balance, \eqref{eq:div_expansion} and multipliying with $\tau$, we deduce the following inequality
	\begin{align*}
		\restrb{\mathcal{E}(\phi_h)}{t^n}{t^{n+1}} &+ \tau\left(\la m(\phi_h^n)\nabla\mu_h^{n+1},\nabla\mu_h^{n+1} \ra_\Omega + \la (\alpha(\phi_h^n)+\beta(\phi_h^n)\snorm{\vv_h^{n+1}}^{s-1})\vv_h^{n+1},\vv_h^{n+1}\ra_\Omega\right)  \\
		&\leq \tau\left(\la (\alpha(\phi_h^n)+\beta(\phi_h^n)\snorm{\vv_h^{n+1}}^{s-1})\vv_h^{n+1}, \u_h^{n+1}\ra_\Omega  + \la \Gamma_\phi(\phi_h^n),\mu_h^{n+1} \ra_\Omega \right.\\
		& \left.\quad- \la \phi_h^n\nabla\mu_h^{n+1},\u_h^{n+1} \ra_\Omega + \gamma\la \phi_h^n\mu_h^{n+1},\u_h^{n+1}\cdot\nvec \ra_{\partial\Omega}\right)\\
		&:= (a) + (b) + (c) + (d).
	\end{align*}
	Using the improved regularity of $\Psi$, see (A5) in \cref{Assumption2}, we then obtain 
	\begin{align}
		C_0(\norm{\phi_h^{n+1}}_{H^1}^2 &+ \la \Psi(\phi_h^{n+1}),1 \ra)_\Omega + \tau\left(m_0\norm{\nabla\mu_h^{n+1}}_{L^2}^2 + \alpha_0\norm{\vv_h^{n+1}}_{L^2}^2 + \beta_0\norm{\vv_h^{n+1}}_{L^{s+1}}^{s+1}\right) \notag\\
		&\leq  \tau((a) + \ldots + (e)) + C_1(\norm{\phi_h^{n}}_{H^1}^2 + \la \Psi(\phi_h^{n}),1 \ra).
		\label{eq:discenergybound}
	\end{align}
	The four terms on the right-hand side are estimated in standard fashion similar to the continuous case using \eqref{eq:divbounddisc}--\eqref{eq:Lqbounddisc}. This procedure yields
	\begin{align}
		(a) &\leq \frac{\alpha_0}{2}\norm{\vv_h^{n+1}}_{L^2}^2 + \frac{\beta_0}{2}\norm{\vv_h^{n+1}}_{L^{s+1}}^{s+1} + C\norm{\u_h^{n+1}}_{L^2}^2 + C'\norm{\u_h^{n+1}}_{L^{s+1}}^{s+1} \notag\\
		& \leq \frac{\alpha_0}{2}\norm{\vv_h^{n+1}}_{L^2}^2 + \frac{\beta_0}{2}\norm{\vv_h^{n+1}}_{L^{s+1}}^{s+1} + C \label{eq:estadisc}\\
		(b) &= \la \Gamma_\phi(\phi_h^n),\mu_h^{n+1}-\la\mu_h^{n+1}, 1\ra_\Omega \ra_\Omega + \la \Gamma_\phi(\phi_h^n), 1\ra_\Omega\la \mu_h^{n+1}, 1\ra_\Omega \label{eq:estcdisc}\\
		& \leq \frac{m_0}{4}\norm{\nabla\mu_h^{n+1}}_{L^2}^2 + C\norm{\Gamma_\phi(\phi_h^n)}_{L^2}^2 + C|\la \Psi'_\av(\phi_h), 1\ra_\Omega|^2 \notag\\
		& \leq \frac{m_0}{4}\norm{\nabla\mu_h^{n+1}}_{L^2}^2 + C + C\norm{\phi_h^{n+1}}_{L^2}^2 + C\norm{\phi_h^{n}}_{L^2}^2, \label{eq:estbdisc} \\ 
		(c) &\leq \frac{m_0}{4}\norm{\nabla\mu_h^{n+1}}_{L^2}^2 + C\norm{\phi_h^n}^2_{L^6}\norm{\u_h^{n+1}}_{L^3} \notag\\
		&\leq \frac{m_0}{4}\norm{\nabla\mu_h^{n+1}}_{L^2}^2 + C\norm{\phi_h^n}^2_{L^6} \notag\\
		(b) + (d) + (e) &= \la \mu_h^{n+1}\nabla\phi_h^n,\u_h^{n+1} \ra_\Omega + \la \phi_h^{n}\mu_h^{n+1},\div(\u_h^{n+1}) \ra_\Omega\notag\\
		&= \la [\mu_h^{n+1}-\bar\mu_h^{n+1}]\nabla\phi_h^n,\u_h^{n+1} \ra_\Omega + \la \bar\mu_h^{n+1}\nabla\phi_h^n,\u_h^{n+1} \ra_\Omega \\
		&\quad+ \la \phi_h^{n}(\mu_h^{n+1}-\bar\mu_h^{n+1}),\div(\u_h^{n+1}) \ra_\Omega + + \la \phi_h^{n}\bar\mu_h^{n+1},\div(\u_h^{n+1}) \ra_\Omega\notag\\
		&\leq C\norm{\nabla\phi_h^n}_0^2 +\frac{m_0}{4}\norm{\nabla\mu_h^{n+1}}_0^2 + C\norm{\nabla\phi_h^n}_0|\bar\mu_h^{n+1}|\notag\\
		&\leq C\norm{\nabla\phi_h^n}_0^2 +\frac{m_0}{4}\norm{\nabla\mu_h^{n+1}}_0^2 + C\norm{\Psi'_\av(\phi_h)}^2_{L^1} \notag\\
		&\leq C\norm{\nabla\phi_h^n}_0^2 +\frac{m_0}{4}\norm{\nabla\mu_h^{n+1}}_0^2 + C\norm{\phi_h^n}^2_{L^2} + C + C\norm{\phi_h^{n+1}}^2_{L^2}\label{eq:estdgammadisc}
	\end{align}
	where we used the estimate \eqref{eq:Lqbounddisc}. Note that \eqref{eq:estdgammadisc} is only used if $\gamma=1$, while \eqref{eq:estcdisc} is used for $\gamma=0$.
	At this point, we combine \eqref{eq:discenergybound} with the estimates above, which gives
	\begin{align*}
		(C_0-C_2\tau)y^{n+1} + \tau D^{n+1} \leq (C_1+\tau C_3)y^{n}  + \tau C_4 
	\end{align*}
	with
	\begin{align*}
		y^k&:=\norm{\phi_h^{k+1}}_{H^1}^2 + \la \Psi(\phi_h^{n+1}),1 \ra_\Omega, \\ 
		D^k&:=\left( m_0\norm{\nabla\mu_h^{k+1}}_{H^1}^2 + \alpha_0\norm{\vv_h^{k+1}}_{L^2}^2 + \beta_0\norm{\vv_h^{k+1}}_{L^{s+1}}^{s+1}\right).   
	\end{align*}
	Finally, we choose $\tau\leq\frac{C_0}{2C_1}:=C_e$ and apply the discrete Gronwall inequality to find the uniform bounds
	\begin{align}
		\max_k \norm{\phi_h^k}_{H^1}^2 + \tau\sum_{k=1}^{n_T} \norm{\mu_h^k}_{H^1}^2 +  \norm{\vv_h^{k}}_{L^2}^2 +  \norm{\vv_h^{k}}_{L^{1+s}}^{1+s} \leq C(\phi_0). \label{eq:uniformhtaubounds}
	\end{align}
	
	\noindent\textbf{Step 3: Pressure estimates.}
	To deduce a uniform estimate of the pressure as in the continuous case, we need to increase the regularity of $\phi_h$. Hence, we estimate the discrete Laplacian using \eqref{eq:pg2} with $\psi_h=\Delta_h\phi_h^{n+1}$, from which we obtain
	\begin{align*}
		\norm{\Delta_h\phi_h^{n+1}}_{L^2}^2 &\leq C\norm{\mu_h^{n+1}}_{L^2}^2 + C\norm{\Psi'_\av(\phi_h)}_{L^2}^2.  
	\end{align*}
	Summation over all time steps, using (A4) from \cref{Assumptions} and the uniform bounds that we have just derived, see \eqref{eq:uniformhtaubounds},  we deduce that
	\begin{align}
		\tau\sum_{k=0}^{n_T-1}\norm{\Delta_h\phi_h^{k+1}}_{L^2}^2 &\leq C\tau\sum_{k=0}^{n_T-1}\norm{\mu_h^{k+1}}_{L^2}^2 + C\tau\sum_{k=0}^{n_T-1}\norm{\Psi'_\av(\phi_h)}_{L^2}^2  \leq C(\phi_0).  \label{eq:disch2}
	\end{align}
	Then we use inf-sup stability \eqref{eq:infsup} and the variational identity \eqref{eq:pg3} to find
	\begin{align*}
		c\norm{p_h^{n+1}}_{L^2} &\leq \inf_{\ww_h\in\Xh}\frac{\la p_h^{n+1},\ww_h \ra_\Omega}{\norm{\ww_h}_{\Xh}} \\
		&\leq\alpha_\infty\norm{\vv_h^{n+1}}_{L^2} + \beta_\infty\norm{\vv_h^{n+1}}_{L^{s+1}}^{(s+1)/s} + (1-\gamma)\norm{\phi_h^n\nabla\mu_h^{n+1}}_{L^2} + \gamma\norm{\nabla\phi_h^n\mu_h^{n+1}}_{L^2}\\
		& \leq \alpha_\infty\norm{\vv_h^{n+1}}_{L^2} + \beta_\infty\norm{\vv_h^{n+1}}_{L^{s+1}}^{(s+1)/s} + \norm{\phi_h^n}_{L^\infty}\norm{\nabla\mu_h^{n+1}}_{L^2} + \norm{\nabla\phi_h^n}_{L^3}\norm{\mu_h^{n+1}}_{H^1}.
	\end{align*}
	Summation over all time steps yields
	\begin{align*}
		\tau\sum_{k=0}^{n_T-1}\norm{p_h^{k+1}}_{L^2} &\leq \tau\sum_{k=0}^{n_T-1}\alpha_\infty\norm{\vv_h^{k+1}}_{L^2}^2 + \beta_\infty\norm{\vv_h^{k+1}}_{L^{s+1}}^{(s+1)/s} + \norm{\phi_h^k}_{L^\infty}^2   + \norm{\nabla\mu_h^{k+1}}_{L^2}^2+ \norm{\nabla\phi_h^k}_{L^3}^2.
	\end{align*}
	In view of the uniform bounds in \eqref{eq:uniformhtaubounds}, we only need to consider the terms involving the $L^\infty$ and the $L^3$ norms. Using the interpolation results of the discrete Laplacian, see \eqref{eq:discreteinftyinterp}, we then obtain 
	\begin{align*}
		\tau\sum_{k=0}^{n_T-1}\norm{\phi_h^k}_{L^\infty}^2 + \norm{\nabla\phi_h^k}_{L^3}^2\leq \tau\sum_{k=0}^{n_T-1}C\norm{\nabla\phi_h^k}_{L^2}^2 + \norm{\Delta_h\phi_h^k}_{L^2}^2  \leq C. 
	\end{align*}
	The above bounds are derived from the uniform bounds \eqref{eq:uniformhtaubounds} and the additional uniform bounds \eqref{eq:disch2}. Thus, this implies that $p_h$ is uniformly bounded in at least $L^1(0,T;L^2(\Omega)).$ \medskip

	\noindent\textbf{Step 4: Existence of discrete solutions.}
	We consider the $(n+1)$-th time step and assume that $\phi_h^{n}$ is already known. 
	After choosing a basis for $\Vh\times\Vh\times\Xh\times\Qh$, we rewrite the variational system \eqref{eq:pg1}--\eqref{eq:pg4} as a nonlinear system in the form of $$J(x)=0 \text{ in } \mathbb{R}^{Z},$$ with $Z=\mathrm{dim(\Vh)}^2\times\mathrm{dim}(\Xh)\times\mathrm{dim}(\Qh)$. In fact, writing $J(x)\cdot x$ is equivalent to testing the corresponding variational identities with $$x=(\mu_h^{n+1},\dtau\phi_h,\vv_h^{n+1},p_h^{n+1}).$$
	As a consequence of the latter, we thus obtain 
	\begin{align*}
		J(x)\cdot x &\geq y^{n+1} - Cy^n + \tau D^{n+1} - \tau C.
	\end{align*}
	Using the same arguments as we have used to derive the a-priori bound \eqref{eq:uniformhtaubounds}, we directly obtain $$J(x)\cdot x \to \infty \text{ for } |x| \to \infty.$$ 
	Thus, the existence of a solution follows from a corollary to Brouwer's fixed-point theorem, see \cite[Proposition~2.8]{Zeidler1}.\medskip

	\noindent\textbf{Step 5: Uniqueness of discrete solutions.}
	It remains to show the uniqueness of discrete solutions satisfying \cref{prob:ac2}. To this end, it is enough to consider one time step and assume that for given $\phi_h^n$ there exist at least two different solutions $(\phi_1^{n+1},\mu_1^{n+1},\vv_1^{n+1},p_1^{n+1})$ and $(\phi_2^{n+1},\mu_2^{n+1},\vv_2^{n+1},p_2^{n+1})$. Then the differences are denoted by $\phi^{n+1}:=\phi_1^{n+1}-\phi_2^{n+1}$, $\mu^{n+1}:=\mu_1^{n+1}-\mu_2^{n+1}$ , $\vv^{n+1}:=\vv_1^{n+1}-\vv_2^{n+1}$ and $p^{n+1}:=p_1^{n+1}-p_2^{n+1}$ and they satisfy the variational equations
	\begin{align}
		\la\phi_h^{n+1},\psi_h\ra_\Omega  - \la \phi_h^{n}\vv^{n+1},\nabla\psi \ra_\Omega + \gamma\la \phi_h^{n}\vv^{n+1}\cdot\nvec,\psi \ra_{\partial\Omega}+ \la m(\phi_h^n)\nabla\mu^{n+1},\nabla\psi_h \ra_\Omega &= 0, \label{eq:pgu1} \\[.2cm]
		\la \mu_h^{n+1}, \xi_h\ra_\Omega - \varepsilon^2\la \nabla\phi^{n+1},\nabla\xi_h \ra_\Omega -  \la \Psi'_\av(\phi_1)-\Psi'_\av(\phi_2),\xi_h \ra_\Omega &= 0, \label{eq:pgu2}\\[.2cm]
		\la (\alpha(\phi^n)\vv_h^{n+1} ,\ww_h\ra_\Omega + \la \beta(\phi^n)(|\vv_1^{n+1}|^{s-1}\vv_1^{n+1}-|\vv_2^{n+1}|^{s-1}\vv_2^{n+1}),\ww_h \ra  - \la p^{n+1},\div(\ww_h) \ra_\Omega& \notag\\
		+ \la \phi^n_h\nabla\mu^{n+1},\ww_h \ra_\Omega+ \gamma\la \phi^n_h\mu^{n+1},\div(\ww_h) \ra_\Omega - \gamma\la \phi_h^{n}\mu^{n+1},\ww_h\cdot\nvec \ra_{\partial\Omega}  &=0, \label{eq:pgu3}\\[.2cm]
		\la \div(\vv^{n+1}),q_h \ra_\Omega  &= 0, \label{eq:pgu4}
	\end{align}
	for every $(\psi_h,\xi_h,\ww_h,q_h)\in\Vh\times\Vh\times\Xh\times\Qh$. To measure the distance between both solutions, we introduce the relative energy
	\begin{align*}
		\mathcal{E}(\phi_1|\phi_2):= \frac{\varepsilon^2}{2}\norm{\nabla\phi}_{L^2}^2 + \la \Psi(\phi_1)-\Psi(\phi_2)-\Psi'(\phi_2)(\phi), 1\ra_\Omega  + \frac{\tilde c}{2}\norm{\phi}_{L^2}^2. 
	\end{align*}
	The constant $\tilde c$ is chosen so that $\mathcal{E}(\phi_1|\phi_2) \geq C\norm{\phi}_{H^1}^2$, see (A5) in \cref{Assumption2}. 
	The aim of this part is to show that $\mathcal{E}(\phi_1^{n+1}|\phi_2^{n+1}) \leq 0.$ For this, we consider the time difference between two time steps and compute
	\begin{align*}
		\frac{1}{\tau}\mathcal{E}(\phi_1^{n+1}|\phi_2^{n+1}) &= \frac{1}{\tau}(\mathcal{E}(\phi_1^{n+1}|\phi_2^{n+1})-\mathcal{E}(\phi_1^{n}|\phi_2^{n})) \\
		&= \varepsilon^2\la \nabla\dtau\phi,\nabla\phi^{n+1} \ra_\Omega + \la \Psi'_\av(\phi_1) - \Psi'_\av(\phi_2),\dtau\phi \ra + \tilde c\la \dtau\phi,\phi^{n+1} \ra_\Omega\\
		&\quad + \la \dtau\phi , \Psi'_\av(\phi_1) - \Psi'_\av(\phi_2) - \Psi''_\av(\phi_2)\phi\ra_\Omega - \tfrac{\tau\varepsilon^2}{2}\norm{\dtau\nabla\phi^{n+1}}_{L^2}^2 - \tfrac{\tilde c\tau}{2}\norm{\dtau\phi^{n+1}}_{L^2}^2 \\
		& = \varepsilon^2\la \nabla\dtau\phi,\nabla\phi^{n+1} \ra_\Omega + \la \Psi'_\av(\phi_1) - \Psi'_\av(\phi_2),\dtau\phi \ra + \tilde c\la \dtau\phi,\phi^{n+1} \ra_\Omega \\
		&\quad + \la \dtau\phi , \Psi'_\av(\phi_1) - \Psi'_\av(\phi_2) - \Psi''_\av(\phi_2)\phi\ra_\Omega - \mathcal{D}_\num^{n+1}\\
		& = (a) + (b) + (c) + (d) + (e).
	\end{align*}
	Here, we introduced the relative numerical dissipation $\mathcal{D}_\num^{n+1}$ as
	\begin{equation} \label{Def:NumDiss}
		\mathcal{D}_\num^{n+1} := \tfrac{\tau\varepsilon^2}{2}\norm{\dtau\nabla\phi^{n+1}}_{L^2}^2 + \tfrac{\tilde c\tau}{2}\norm{\dtau\phi^{n+1}}_{L^2}^2.  
	\end{equation}
	In the following, we either expand or estimate all the terms appearing. By inserting $\xi_h=\dtau\phi$ in \eqref{eq:pgu1} and $\psi_h=\mu^{n+1}$ in \eqref{eq:pgu2}, we expand $(a)$ and $(b)$ to
	\begin{align*}
		(a) + (b) = -\la m(\phi_h^n)\nabla\mu^{n+1},\nabla\mu^{n+1} \ra_\Omega + \la \phi_h^n\vv^{n+1},\nabla\mu^{n+1}\ra_\Omega - \gamma\la \phi_h^{n}\vv^{n+1}\cdot\nvec,\mu^{n+1} \ra_{\partial\Omega}.
	\end{align*}
	Furthermore, choosing $\ww_h=\vv^{n+1}$ in \eqref{eq:pgu3} and $q_h=p^{n+1}$ in \eqref{eq:pgu4}, it yields
	\begin{align*}
		\la\alpha(\phi_h^n)\vv^{n+1},\vv^{n+1}\ra_\Omega + \la \beta(\phi_h^n)(\snorm{\vv_1^{n+1}}\vv_1^{n+1}-\snorm{\vv_2^{n+1}}\vv_2^{n+1}), \vv^{n+1}\ra_\Omega + \la \phi_h^n\nabla\mu^{n+1},\vv^{n+1} \ra_\Omega = 0. 
	\end{align*}
	Then, applying inequality \eqref{ineq.2} shows 
	\begin{align*}
		(a) + (b) \leq -m_0\norm{\nabla\mu^{n+1}}_{L^2}^2 - \alpha_0\norm{\vv^{n+1}}_{L^2}^2 - 2^{1-s}\beta_0\norm{\vv^{n+1}}_{L^{s+1}}^{s+1}.
	\end{align*}
	For the term $(c)$, we select the test function $\psi_h=\phi^{n+1}$ in \eqref{eq:pgu1} to deduce that
	\begin{align*}
		(c) &=  -\tilde c\la m(\phi_h^n)\nabla\mu^{n+1},\nabla\phi^{n+1} \ra_\Omega + \tilde c\la \phi_h^n\vv^{n+1},\nabla\phi^{n+1}\ra_\Omega - \tilde c\gamma\la \phi_h^{n}\vv^{n+1}\cdot\nvec,\phi^{n+1} \ra_{\partial\Omega} \\
		& \leq \frac{m_0}{4}\norm{\nabla\mu^{n+1}}_{L^2}^2 + \tilde c\la \phi_h^n\vv^{n+1},\nabla\phi^{n+1}\ra_\Omega - \tilde c\gamma\la \phi_h^{n}\vv^{n+1}\cdot\nvec,\phi^{n+1} \ra_{\partial\Omega}
	\end{align*}
	Next, we estimate the remaining terms in $(c)$. We make a case distinction and first consider $\gamma=0$. We apply the discrete interpolation inequality, see \eqref{eq:discreteinftyinterp}, to deduce that
	\begin{align*}
		\la \phi_h^n\vv^{n+1},\nabla\phi^{n+1}\ra_\Omega &\leq  \norm{\vv^{n+1}}_{L^2}\norm{\nabla\phi^{n+1}}_{L^3}\norm{\phi_h^n}_{L^6} \\
		&\leq \frac{\alpha_0}{2}\norm{\vv^{n+1}}_{L^2}^2 + C\norm{\nabla\phi^{n+1}}_{L^3}^2\\
		&\leq  \frac{\alpha_0}{2}\norm{\vv^{n+1}}_{L^2}^2 + C\norm{\nabla\phi^{n+1}}_{L^2}^2 + \delta\norm{\Delta_h\phi^{n+1}}_{L^2}^2.
	\end{align*}
	Furthermore, in the case of $\gamma=1$, we obtain
	\begin{align*}
		\la \phi_h^n\vv^{n+1},\nabla\phi^{n+1}\ra_\Omega - \gamma\la \phi_h^{n}\vv^{n+1}\cdot\nvec,\phi^{n+1} \ra_{\partial\Omega} &= - \la \phi^{n+1}\vv^{n+1},\nabla\phi_h^{n} \ra_\Omega   \\
		&\quad\leq \norm{\vv^{n+1}}_{L^2}\norm{\nabla\phi^{n+1}}_{L^2}\norm{\phi_h^n}_{L^\infty} \\
		&\quad\leq \frac{\alpha_0}{2}\norm{\vv^{n+1}}_{L^2}^2 + C\norm{\nabla\phi^{n+1}}_{L^2}^2 + \delta\norm{\Delta_h\phi^{n+1}}_{L^2}^2.
	\end{align*}
	Hence, we consider the discrete Laplacian and deduce that
	\begin{align*}
		\norm{\Delta_h\phi^{n+1}}_{L^2}^2 &\leq C\norm{\mu^{n+1}}_{L^2}^2 + C\norm{\Psi'_\av(\phi^{n+1}_1)-\Psi'_\av(\phi^{n+1}_2)}_{L^2}^2 \\
		&\leq C\norm{\nabla\mu^{n+1}}_{L^2}^2 + C|\bar\mu^{n+1}|^2+ C\norm{\Psi'_\av(\phi^{n+1}_1)-\Psi'_\av(\phi^{n+1}_2)}_{L^2}^2 \\
		& \leq C\norm{\nabla\mu^{n+1}}_{L^2}^2 + C\norm{\Psi''_\av(\xi^{n+1})\phi^{n+1}}_{L^2}^2 \\
		& \leq C\norm{\nabla\mu^{n+1}}_{L^2}^2 + C\norm{\phi^{n+1}}_{H^1}^2.
	\end{align*}
	Choosing now $\delta = \frac{m_0}{4C}$, we find in total
	\begin{align*}
		(c) & \leq \frac{m_0}{2}\norm{\nabla\mu^{n+1}}_{L^2}^2 + \frac{\alpha_0}{2}\norm{\vv^{n+1}}_{L^2}^2 + C\norm{\phi^{n+1}}_{H^1}^2.  
	\end{align*}
	For the last term $(d)$, we use the numerical dissipation \cref{Def:NumDiss} and the growth estimates of $\Psi$, see (A5) in \cref{Assumption2}, to estimate
	\begin{align*}
		(d) &= \la \Psi'''_\av(\zeta)\phi^{n+1}, \dtau\phi^{n+1}(\phi_2^{n+1}-\phi_n^{n+1})\ra \\
		&\leq  \norm{\Psi'''_\av(\zeta)}_{L^6}\norm{\phi^{n+1}}_{L^6}(\norm{\phi_2^{n+1}}_{L^6} + \norm{\phi_2^{n}}_{L^6})\norm{\dtau\phi^{n+1}}_{L^6} \\
		&\leq \delta\norm{\dtau\phi^{n+1}}_{H^1}^2 +  C\norm{\Psi'''_\av(\zeta)}_{L^6}^2\norm{\phi^{n+1}}_{L^6}^2(\norm{\phi_2^{n+1}}_{L^6}^2 + \norm{\phi_2^{n}}_{L^6}^2 ) \\
		&\leq \delta\norm{\dtau\phi^{n+1}}_{L^2}^2 + C\norm{\phi^{n+1}}_{H^1}^2.    
	\end{align*}
	Finally, combining all previous estimates and using that $\mathcal{E}^n=0$, it yields
	\begin{align*}
		\mathcal{E}(\phi_1^{n+1}|\phi_2^{n+1}) &+ \frac{\tau}{2}\left(m_0\norm{\nabla\mu^{n+1}}_{L^2}^2 + \alpha_0\norm{\vv^{n+1}}_{L^2}^2 + 2^{1-s}\beta_0\norm{\vv^{n+1}}_{L^{s+1}}^{s+1} \right) \leq \tau C_1\mathcal{E}(\phi_1^{n+1}|\phi_2^{n+1})
	\end{align*}
	Under the assumed time step condition $\tau \leq \frac{1}{2C_1}:=C_u$, we find that
	\begin{align*}
		\frac{1}{2}\mathcal{E}(\phi_1^{n+1}|\phi_2^{n+1}) &+ \frac{\tau}{2}\left(m_0\norm{\nabla\mu^{n+1}}_{L^2}^2 + \alpha_0\norm{\vv^{n+1}}_{L^2}^2 + 2^{1-s}\beta_0\norm{\vv^{n+1}}_{L^{s+1}}^{s+1}\right) \leq 0.
	\end{align*}
	This estimate yields $\phi^{n+1}=0$, $\vv^{n+1}=0$ and $\nabla\mu^{n+1}=0$. Furthermore, using \eqref{eq:pgu2} with $\xi_h=1$, it shows that $\mu^{n+1}=0$. Lastly, we apply the inf-sup stability \eqref{eq:infsup} and \eqref{eq:pgu3} to obtain
	\begin{equation*}
		\norm{p^{n+1}}_{L^2} \leq \inf_{\ww_h\in\Xh} \frac{\la p^{n+1},\div(\ww_h) \ra_\Omega}{\norm{\ww_h}_{\Xh}} = 0.  
	\end{equation*}
	Thus, we conclude that both solutions coincide.
\end{proof}

\section{Numerical simulations} \label{Sec:Simul}
The scheme of the previous section (see \cref{prob:ac2}) is implemented in NGSolve \cite{schberl2014c++11}. The nonlinear system is treated using Newton's method with an absolute tolerance of $10^{-11}$, while the linear system is solved by a direct solver. We consider the two-dimensional domain $\Omega = [0,1]^2$ and denote by $(x,y)$ a point in the domain. In the experiments, we fix $\eps^2 = 10^{-4}$ and use the following functions:
\begin{align*}
	\Psi(\phi) &= \frac{1}{4}\phi^2(1-\phi)^2, & m(\phi)&=10^{-2} + \phi^2(1-\phi)^2\max\{0,1-\phi^2\}, \\
	\alpha(\phi) &= \alpha_1\phi+\alpha_2(1-\phi), & \beta(\phi) &= \beta_1\phi+\beta_2(1-\phi),
\end{align*}
where $\alpha(\phi)$ and $\beta(\phi)$ are extended outside $[0,1]$ by their constant upper or lower limit. With the above choice of the potential function, its time average can be exactly computed as
\begin{equation*}
	\Psi'_{av}(\phi_h):= \frac{1}{6}\left( \Psi'(\phi_h^{n+1}) + \frac{4}{6}\Psi'\left(\frac{\phi_h^{n+1}+\phi_h^{n}}{2}\right)+\Psi'(\phi_h^{n})\right).  
\end{equation*}

\subsection{Convergence test and inconsistent volume source}
We first consider the parameters $\alpha_1=10^{-2}$, $\alpha_2=10^0$, $\beta_1=10^{-1}$, $\beta_2=10^1$, along with the initial condition and parametric functions
\begin{align}
	\phi_0(x,y) &= \frac{1}{2} - \frac{1}{2}\tanh\left(\frac{\sqrt{1.1(x-0.5)^2 + 0.8(y-0.5)^2}-0.25}{\sqrt{10\varepsilon^2}}\right), \label{eq:exp1}\\
	\Gamma_\phi(\phi)&=\frac{1}{5}\phi\max\{0,1-\phi^2\}, \qquad  \Gamma_\vv(\phi) =2\phi\max\{0,1-\phi^2\}. \notag
\end{align}

\noindent\textbf{Convergence test:}  
Since no exact solution is available, we compute the error using a refined solution for a fixed $\tau = 5\times10^{-3}$. The error quantities are defined as
$$\begin{alignedat}{2}
	\text{err}(\phi,h_k)&:= \max_{n\in\Itau}\norm{\phi^n_{{h_k}}-\phi^n_{h_{k+1}}}_{H^1}^2, \hspace{1.2cm} 
	\text{err}(\mu,h_k)&:=\tau\sum_{n=1}^{n_T}\norm{\mu^n_{h_k}-\mu^n_{h_{k+1}}}_{H^1}^2 , \\
	\text{err}(\vv,h_k)&:=\tau\sum_{n=1}^{n_T}\norm{\vv^n_{{h_k}}-\vv^n_{h_{k+1}}}_{L^{2}_\div}^2,  \hspace{1cm}
	\text{err}(p,h_k)&:=\tau\sum_{n=1}^{n_T}\norm{p_{h_k}^n-p^n_{h_{k+1}}}_{L^2}^2,
\end{alignedat}$$
where mesh refinements use $h_k \approx 2^{-1-k}$ for $k=0,\ldots,6$. \cref{table1} and \cref{table2} present the errors and squared experimental orders of convergence (eoc) for $\gamma=0$ and $\gamma=1$, respectively. As expected, we observe first-order convergence for $(\phi, \mu)$ and second-order convergence for $(\vv, p)$. These rates are order-optimal, confirming the robustness of the scheme despite nonlinear coupling in the Forchheimer subsystem. 

\begin{table}[H]
	\centering
	\small
	\caption{Energy errors and squared experimental order of convergence (eoc) for $s=3$ and $\gamma=1$; first-order convergence for $(\phi, \mu)$ (theoretically optimal in $H^1$), second-order convergence for $(\vv, p)$.} 
	\begin{tabular}{|c||c|c|c|c|c|c|c|c|}
		\hline
		$ k $ & $\text{err}(\phi,h_k)$ &  eoc & $\text{err}(\mu,h_k)$ & eoc & $\text{err}(\vv,h_k)$ &  eoc & $\text{err}(p,h_k)$ & eoc   \\
		\hline
		1& $5.50\cdot 10^{+0}$ &  $\phm0.00$ &  $3.65\cdot 10^{-2}$ &  $\phm0.00$ &  $6.59\cdot 10^{-1}$ &  $0.00$ &  $2.91\cdot 10^{-3}$ &  $0.00$ \\
		2&$1.54\cdot 10^{+1}$ &  $-1.49$ &  $8.45\cdot 10^{-2}$ &  $-1.21$ &  $2.14\cdot 10^{-1}$ &  $1.62$ &  $7.69\cdot 10^{-4}$ &  $1.92$ \\  
		3&$3.40\cdot 10^{+1}$ &  $-1.14$ &  $9.72\cdot 10^{-2}$ &  $-0.20$ &  $1.29\cdot 10^{-1}$ &  $0.73$ &  $8.87\cdot 10^{-5}$ &  $3.12$ \\
		4&$1.31\cdot 10^{+1}$ &  $\phm1.38$ &  $3.47\cdot 10^{-2}$ &  $\phm1.49$ &  $2.22\cdot 10^{-2}$ &  $2.54$ &  $5.90\cdot 10^{-6}$ &  $3.91$ \\
		5&$2.79\cdot 10^{+0}$ &  $\phm2.23$ &  $7.48\cdot 10^{-3}$ &  $\phm2.21$ &  $2.19\cdot 10^{-3}$ &  $3.35$ &  $5.01\cdot 10^{-7}$ &  $3.56$ \\
		6&$5.16\cdot 10^{-1}$ &  $\phm2.43$ &  $1.57\cdot 10^{-3}$ &  $\phm2.25$ &  $1.37\cdot 10^{-4}$ &  $4.00$ &  $3.16\cdot 10^{-8}$ &  $3.99$ \\
		\hline
	\end{tabular}
	\label{table1}
\end{table}

\begin{table}[H]
	\centering
	\small
	\caption{Energy errors and squared experimental order of convergence (eoc) for $s=3$, $\gamma=1$; rates match $\gamma=0$ case, confirming boundary condition robustness.} 
	\begin{tabular}{|c||c|c|c|c|c|c|c|c|}
		\hline
		$ k $ & $\text{err}(\phi,h_k)$ &  eoc & $\text{err}(\mu,h_k)$ & eoc & $\text{err}(\vv,h_k)$ &  eoc & $\text{err}(p,h_k)$ & eoc   \\
		\hline
		1 &$5.50\cdot 10^{+0}$ &  $\phm0.00$ &  $3.61\cdot 10^{-2}$ &  $\phm0.00$ &  $6.41\cdot 10^{-1}$ &  $0.00$ &  $3.05\cdot 10^{-3}$ &  $0.00$ \\
		2 &$7.26\cdot 10^{+0}$ &  $-0.40$ &  $8.46\cdot 10^{-2}$ &  $-1.23$ &  $1.94\cdot 10^{-1}$ &  $1.73$ &  $6.13\cdot 10^{-4}$ &  $2.32$ \\
		3 &$2.53\cdot 10^{+1}$ &  $-1.80$ &  $8.85\cdot 10^{-2}$ &  $-0.06$ &  $1.18\cdot 10^{-1}$ &  $0.71$ &  $2.91\cdot 10^{-5}$ &  $4.40$ \\
		4 &$1.23\cdot 10^{+1}$ &  $\phm1.03$ &  $2.86\cdot 10^{-2}$ &  $\phm1.63$ &  $2.19\cdot 10^{-2}$ &  $2.44$ &  $1.66\cdot 10^{-6}$ &  $4.13$ \\
		5 &$2.71\cdot 10^{+0}$ &  $\phm2.19$ &  $6.28\cdot 10^{-3}$ &  $\phm2.19$ &  $2.17\cdot 10^{-3}$ &  $3.34$ &  $1.27\cdot 10^{-7}$ &  $3.71$ \\
		6 &$5.00\cdot 10^{-1}$ &  $\phm2.44$ &  $1.29\cdot 10^{-3}$ &  $\phm2.28$ &  $1.36\cdot 10^{-4}$ &  $4.00$ &  $7.69\cdot 10^{-9}$ &  $4.04$ \\
		\hline
	\end{tabular}
	\label{table2}
\end{table}

\noindent\textbf{Simulations:}  
We numerically perform the above experiment for $s \in \{1,2,3\}$ and $\gamma\in\{0,1\}$. The results are depicted in \cref{fig:exp1} for $\gamma=0$ and \cref{fig:exp1b} for $\gamma=1$. In all experiments, the initial ellipse expands toward the boundary of the domain, while its core breaks up due to the volume source. This results in propagating interface fronts until the $\phi=1$ phase reaches the boundary. At this stage, the solutions for $\gamma=0$ and $\gamma=1$ diverge fundamentally. For $\gamma=1$, see \cref{fig:exp1b}, the propagation front continues until the domain is dominated by the $\phi=0$ phase. However, for $\gamma=1$, see \cref{fig:exp1}, the phase-field exceeds its physical range, forming a boundary layer with $\phi\geq 1$. This "blowup" phenomenon suggests that the system with Neumann-like boundary conditions imposes additional constraints on the choice of $\Gamma_\vv$.

The Forchheimer nonlinearity ($s>1$) also influences the dynamics of the phase-field: for $s=1$ (Darcy flow), the evolution is faster due to weaker velocity damping, while higher $s$ ($s=2,3$) slows propagation.
\begin{figure}[H]
	\centering
	\includegraphics[page=1,width=.9\textwidth]{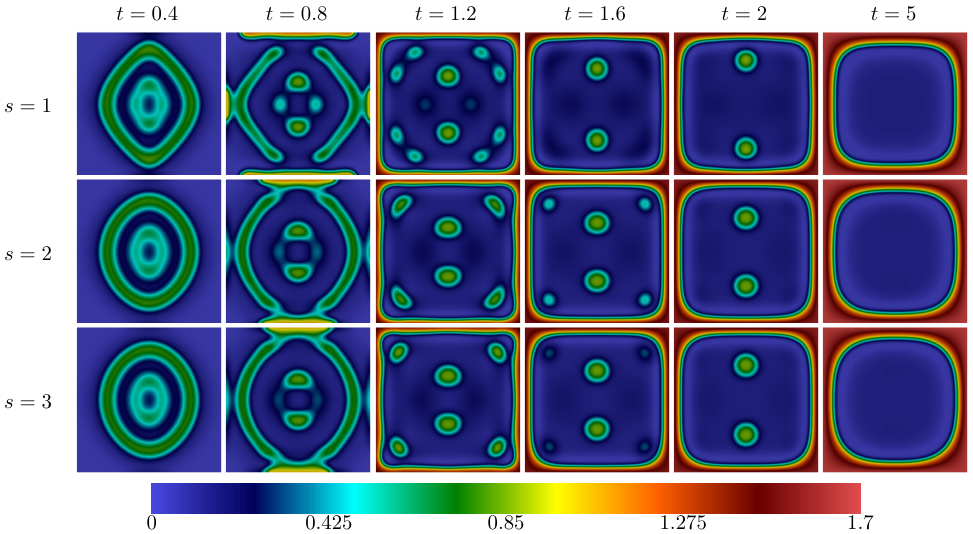}
	\caption{Phase-field $\phi$ evolution with $\gamma=0$ for $s \in \{1,2,3\}$ (top to bottom) at times $t\in\{0.4,0.8,1.2,1.6,2,5\}$ (left to right); mass leakage via Neumann-like conditions causes unphysical boundary layers with $\phi\geq1$.
	}
	\label{fig:exp1}
\end{figure}

\begin{figure}[H]
	\centering
\includegraphics[page=2,width=.9\textwidth]{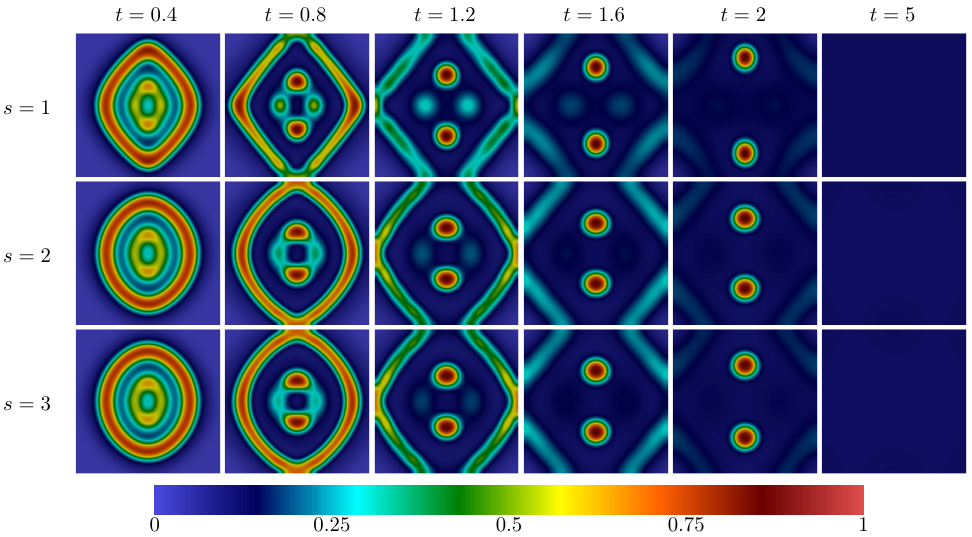}
	\caption{Phase-field $\phi$ evolution with $\gamma=1$ for $s \in \{1,2,3\}$ (top to bottom) at times $t\in\{0.4,0.8,1.2,1.6,2,5\}$ (left to right); boundary conditions enable mass influx, driving domain invasion.}
	\label{fig:exp1b}
\end{figure}

\Cref{fig:meta1} shows the evolution of mass and energy, while \cref{fig:meta1error} confirms the structure-preserving properties of the numerical scheme. For $\gamma=0$, the phase-field's mass increases strictly, while for $\gamma=1$, the mass decreases after a plateau due to boundary leakage. The energy dissipation error remains non-positive, aligning with theoretical predictions. Mass conservation balance is preserved up to machine precision, validating the exact mass balance of the scheme.

\begin{figure}[H]
	\centering
\includegraphics[page=3,width=.9\textwidth]{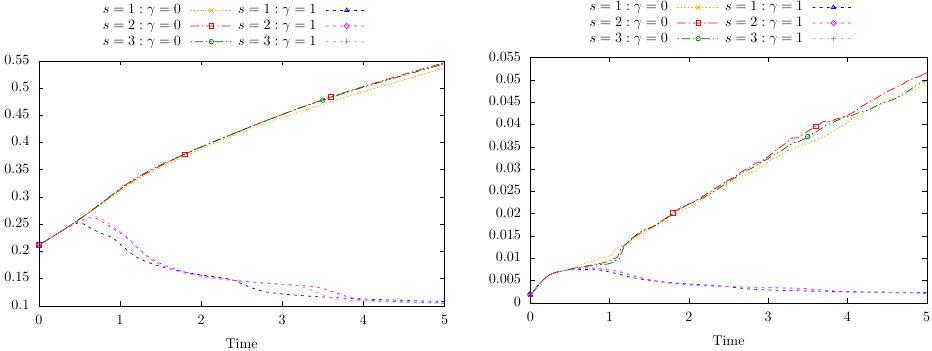}
	\caption{Evolution of the mass (left) and energy (right);  
		strict mass increase ($\gamma=0$) vs. leakage ($\gamma=1$) and energy decay aligns with theoretical dissipation.}
	\label{fig:meta1}
\end{figure}

\begin{figure}[H]
	\centering
\includegraphics[page=4,width=.9\textwidth]{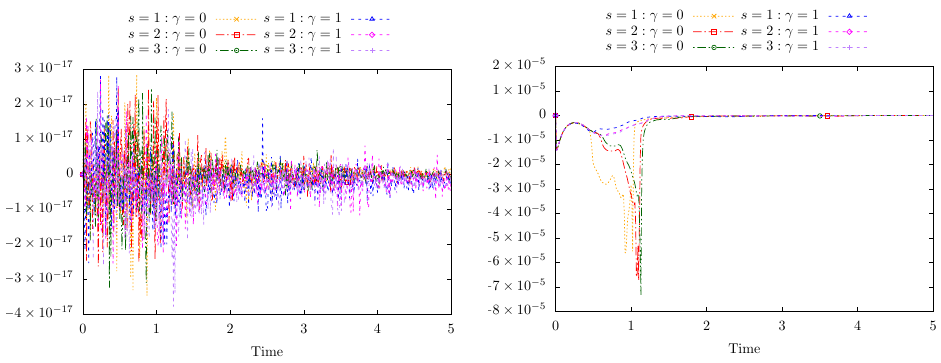}
	\caption{Structure-preserving properties of the scheme; machine-precision mass balance errors (left) and non-positive energy dissipation error (right), confirming discrete entropy stability.}
	\label{fig:meta1error}
\end{figure}

\subsection{Consistent volume source}
To address boundary-driven blowup, we consider a consistent mass source with initial data and parametric functions
\begin{align}
	\phi_0(x,y) &= \frac{1}{2} - \frac{1}{2}\tanh\left(\frac{\sqrt{1.1(x-0.5)^2 + 0.8(y-0.75)^2}-0.25}{\sqrt{10\varepsilon^2}}\right), \label{eq:exp2}\\
	\Gamma_\phi&=0, \qquad  \Gamma_\vv=\phi-\la \phi_0,1\ra_\Omega. \notag
\end{align}

\Cref{fig_meta2} shows mass and energy evolution, while \cref{fig_meta2error} confirms the structure-preserving properties. The mass balance error is at machine precision, and the energy dissipation error remains non-positive for all cases. In the case of $\gamma=0$ as expected mass is conserved.

\begin{figure}[H]
	\centering
	\includegraphics[page=6,width=.9\textwidth]{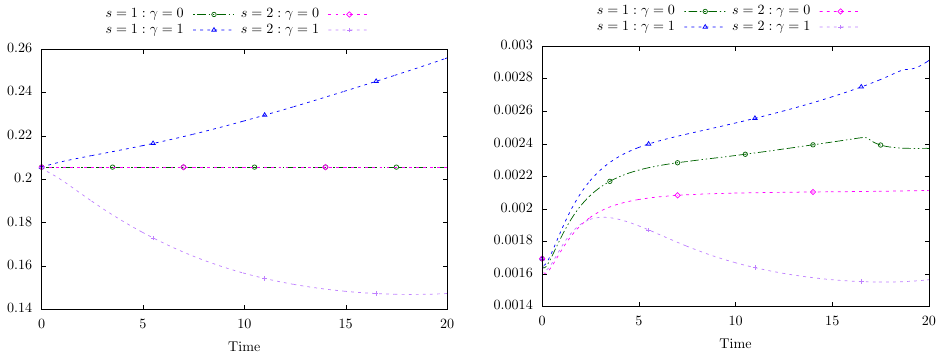}
	\caption{Mass (left) and energy evolution (right) for consistent source; strict conservation for $\gamma=0$ vs. source-driven growth/decay for $\gamma=1$ (left)
		and energy dissipation remains monotonic (right).}
	\label{fig_meta2}
\end{figure}

\begin{figure}[H]
	\centering
	\includegraphics[page=7,width=.8\textwidth]{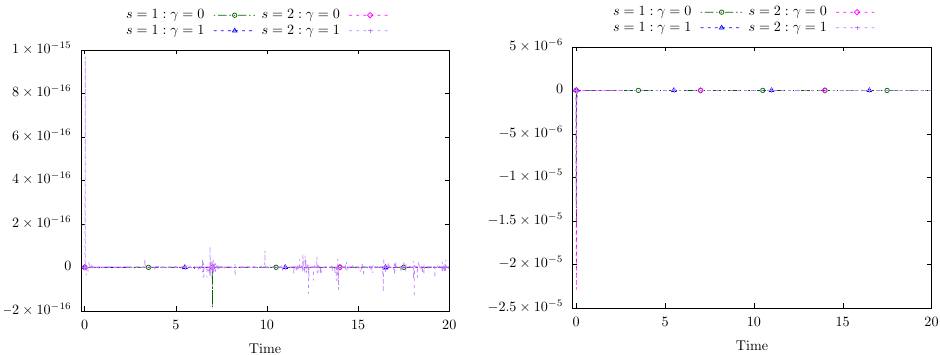}
	\caption{Error balances for consistent source; mass errors at machine precision (left) and  energy dissipation errors $\leq 0$ (right), reinforcing scheme robustness.}
	\label{fig_meta2error}
\end{figure}

The simulation results are depicted in \cref{fig:exp2}. For $\gamma=0$, the initial ellipse that touches the boundary evolves into a droplet that detaches and migrates toward the center of the domain. However, for $\gamma=1$, the droplet remains attached to the boundary, forming a secondary interface. This contrast arises because $\Gamma_\vv$ is mean-free only for $\gamma=0$, ensuring mass conservation. For $\gamma=1$, mass leakage persists, but the phase-field remains within $[0,1]$, avoiding blow-up. 

\begin{figure}[H]
	\centering
\includegraphics[page=5,width=.85\textwidth]{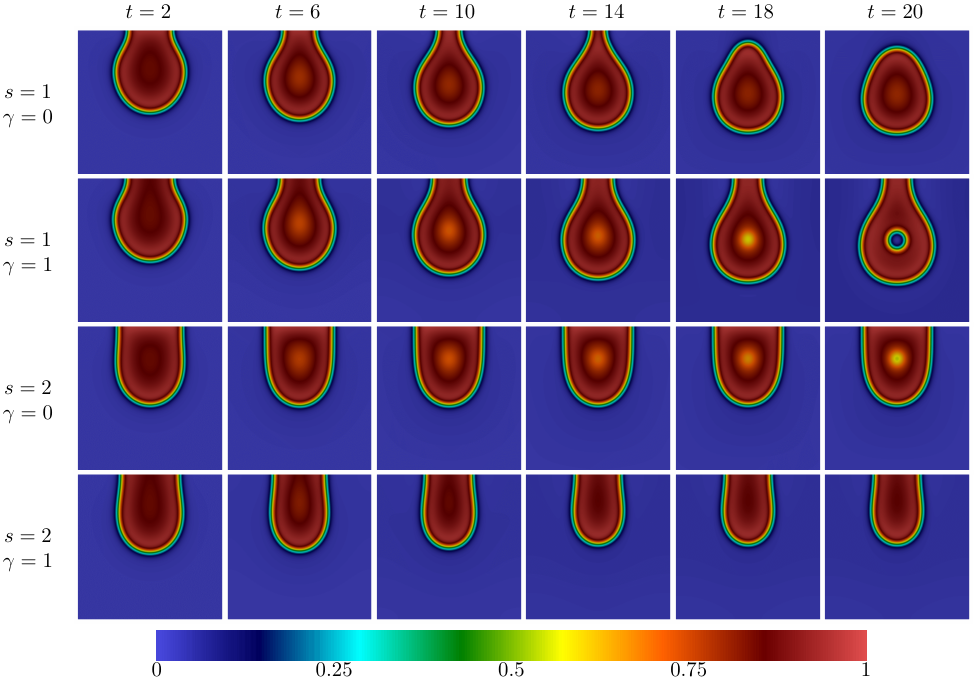}
	\caption{Phase-field $\phi$ with consistent mass source at times $t\in\{2,6,10,14,18,20\}$ (left to right) in the cases $s \in \{1,2\}$ and $\gamma\in \{0,1\}$ as indicated on the left column.}
	\label{fig:exp2}
\end{figure}

\section{Conclusion and outlook}

We analyzed a coupled Cahn--Hilliard--Forchheimer system with concentration-dependent mobility, nonlinear transport, and solution-dependent mass/volume sources. Key achievements include a new well-posedness result for the generalized Forchheimer subsystem, the existence of weak solutions for the fully coupled system through energy-controlled Galerkin approximations and nonlinear compactness techniques, a fully discrete scheme that maintains exact mass conservation and discrete energy dissipation, and the numerical validation of optimal convergence rates and geometric structure inheritance.

In future work, we would like to study the case of degenerate mobility and singular potentials for both boundary conditions $\gamma \in \{0,1\}$. In particular, we are interested in the restrictions that arise for the volume source term. The numerical simulations indicate that in certain configurations additional constraints on the volumetric source may arise. To retain a phase-dependent source term, this might result in nonlocal source terms.

\appendix

\section{Proof of \cref{Lem:Forchheimer}}
\label{Appendix}

\noindent\textbf{Step 0: Key inequalities.} For all vectors $x,y \in \mathbb{R}^d$ and $s>1$, we have:
\begin{align}
	||x|^{s-1}x - |y|^{s-1}y| &\leq s 2^{s-1} (|x|^{s-1} + |y|^{s-1})|x - y| \label{ineq.1} \\
	(|x|^{s-1}x - |y|^{s-1}y)\cdot(x - y) &\geq 2^{1-s}|x - y|^{s+1} \label{ineq.2} \\
	||x|^{(1-s)/s}x - |y|^{(1-s)/s}y| &\leq 2^{1-1/s}|x - y|^{1/s} \label{ineq.3} \\
	\frac{|x - y|^2}{|x|^{1-1/s} + |y|^{1-1/s}} &\leq s \left(\frac{x}{|x|^{1-1/s}} - \frac{y}{|y|^{1-1/s}}\right)(x - y) \label{ineq.4}
\end{align}
These inequalities are standard in the analysis of the generalized Forchheimer equation, see also \cite{fabrie1989regularity,kieu2020existence}. The first inequality can be proved by applying the mean value theorem to $t \mapsto |t|^{s-1}t$. The second follows by convexity of $t \mapsto |t|^{s+1}$. The third by Hölder continuity of $t \mapsto |t|^{(1-s)/s}t$. The fourth follows by an application of the Cauchy--Schwarz inequality. \medskip

\noindent\textbf{Step 1: Variational form.} Let $V := L^{1+s}(\Omega)^d$, $V_\div := L^{1+s}_\div(\Omega)^d$, and $Q := L^{1+1/s}(\Omega)$. We want to find a pair $(\vv,p) \in V_\div \times Q$ such that
\begin{equation} 			\label{mix.form.stat.pr}
	\begin{array}{rcl@{\quad}l}
		a(\vv,\u) - b(\u,p) & = & F(\u) 
		& \mbox{for all } \u \in V_\div , \\[0.5ex]
		b(\vv,q) & = & G(q) 
		& \mbox{for all } q \in Q,
	\end{array}
\end{equation} 
where we introduced the continuous linear forms 
$F: V_\div \to \R$ and $G: Q \to \R$, the bilinear form 
$b: V_\div \times Q \to \R$
and the nonlinear form 
$a: V \times V 
\to \R$
by
$$\begin{aligned} F(\u) &:= (\ff,\u)_\Omega - (h,\u \cdot \nvec)_{\p\Omega} \quad
	&&\mbox{for } \u \in V_\div, \\
	G(q) &:= (g,q)_\Omega \quad
	&&\mbox{for } q \in Q  , \\
	b(\u,q) &:= (\div(\u),q)_\Omega
	&&\mbox{for } \u \in V_\div  , q \in Q, \\
	a(\vv,\u) &:= ( \alpha + \beta |\vv|^{s-1} ,
	\vv \cdot \u)_\Omega 
	&&\mbox{for } \vv,\u \in V . \end{aligned}$$
We note that the trace operator $\gamma: \u \mapsto \u \cdot \nvec$ is indeed linear and continuous from $V_\div$ to $(W^{1/(s+1),1+1/s}(\p\Omega))'$ by an extension argument.
Since $a$ is linear with respect to its second variable,
we can define a mapping $A: V \to V'$ by 
$\langle Av , u \rangle = a(\vv,\u)$.
To show continuity of $A$, we consider 
$\vv_1,\vv_2 \in V$. 
Applying H\"older's inequality we obtain 
$$ | \langle A\vv_1 - A\vv_2 , \u \rangle | \lesssim
( \|\vv_1-\vv_2\|_V
+  
\| |\vv_1|^{s-1}\vv_1-|\vv_2|^{s-1}\vv_2 \|_Q
) \|\u\|_V $$
for all $\u \in V$,
which proves
\begin{equation}			\label{norm(Au)}
	\| A\vv \|_{V'} \lesssim \| \vv\|_V + 
	\| \vv\|_V^s  .
\end{equation}
Applying inequality (\ref{ineq.1}) and again H\"older's inequality yields
$$ \| A \vv_1 - A \vv_2 \|_{V'} \lesssim
( 1 + 
\|\vv_1\|_V^{s-1} + \|\vv_2\|_V^{s-1}  )
\| \vv_1 - \vv_2 \|_V  . $$
Using inequality (\ref{ineq.2}) we obtain
\begin{equation}				\label{A:mon}
	\langle A\vv - Au , \vv -\u \rangle \gtrsim
	\|\vv -\u\|_V^{s+1} \quad 
	\mbox{for} \quad \vv, \u \in V.
\end{equation}
Therefore, $A$ is strictly monotone on $V_\div$, 
and uniformly monotone (and thus coercive) for $V$. \bigskip

\noindent\textbf{Step 2: Regularization.}
For $n > 0$, we define a regularization by looking for a tuple $(\vv_n,p_n) \in V_\div \times Q$ such that
\begin{equation}			\label{reg.stat.pr}
	\begin{array}{rcl}
		a(\vv_n,\u) + d_n(\vv_n,\u) - b(\u,p_n)
		& = & F(\u) \quad \mbox{for all } u \in V_\div  , \\	
		c_n(p_n,q) + b(\vv_n,q) & = & G(q)
		\quad \mbox{for all } q \in Q,
	\end{array}
\end{equation} where the nonlinear forms $d_{n}: V_\div \times V_\div \to \R$
and $c_n: Q \times Q \to \R$ are given by
$$\begin{aligned} d_{n} (\vv,\u) &:= 
	\frac{1}{n} (|\Div(\vv)|^{s-1} \Div(\vv), \Div(\u))_\Omega
	&&\mbox{ for } \vv, \u \in V_\div, \\
	c_n(p,q) &:= 
	\frac{1}{n} (|p|^{(1-s)/s}p,q)_\Omega
	&&\mbox{ for } p,q \in Q . 
\end{aligned}$$

We observe that the operator $A_n : V_\div \to (V_\div)'$, $A_n \vv_n:=a(\vv_n,\cdot)+d_n(\vv_n,\cdot)$, is continuous, coercive 
and strictly monotone on $V_\div$. Indeed,
$$\begin{aligned} &| d_n(\vv_1,\u) - d_n(\vv_2,\u) | \\&\le
	\frac{s2^{s-1}}{n} ( \| \Div(\vv_1) \|_V^{s-1} 
	+ \| \Div(\vv_2) \|_V^{s-1} )
	\| \Div(\vv_1) - \Div(\vv_2) \|_V
	\| \Div(\u) \|_V
\end{aligned}$$
for all $\vv_1,\vv_2,\u \in V_\div$, which is obtained by means 
of H\"older's inequality and of (\ref{ineq.1}).
Furthermore, an application of (\ref{ineq.2}) yields
$$ d_n(\vv,\vv-\u) - d_n(\u,\vv-\u) \ge
\frac{2^{1-s}}{n} \| \Div(\vv - \u) \|_V^{s+1} 
\quad \mbox{for all } \vv, \u \in V_\div  . $$
Together with (\ref{A:mon}) this inequality implies 
the uniform monotonicity of $A_n$:
$$ \langle A_n \vv - A_n \u , \vv - \u \rangle_{V \times V'}
\ge 2^{1-s} \min \left( C(\beta_0) , \frac{1}{n} \right) 
\| \vv-\u \|_V^{s+1} \quad \mbox{for all } \vv, \u \in V  . $$
The coercivity and strict monotonicity of $A_n$ are direct consequences
of the uniform monotonicity.

Furthermore, the operator $C_n : Q \to Q'$, $C_n p_n:=c_n(p_n,\cdot)$, is continuous, coercive
and strictly monotone on $Q$. Indeed, using H\"older's inequality and (\ref{ineq.3}) we obtain
the continuity of $C_n$:
$$ \| C_n p - C_n q \|_{Q'} \le
\frac{1}{n} \left\| \frac{p}{|p|^{1-1/s}} - \frac{q}{|q|^{1-1/s}} 
\right\|_V \le
\frac{2^{1-1/s}}{n} \| p - q \|_Q^{1/s}
\quad \mbox{for all } p, q \in Q  . $$
The coercivity follows from the equation
$$ \langle C_n q , q \rangle_{Q' \times Q} = 
\frac{1}{n} (q^2, |q|^{(1-s)/s})_\Omega =
\frac{1}{n} \|q\|_Q^{(s+1)/s}  , $$
and strict monotonicity 
from the strict monotonicity of $x \mapsto |x|^{-1+1/s} x$:
$$ \langle C_n p - C_n q , p - q \rangle_{Q' \times Q} = 
\frac{1}{n} \int_\Omega 
\left( \frac{p}{|p|^{1-1/s}} - \frac{q}{|q|^{1-1/s}} \right) (p-q) \, dx
> 0 $$
for all $p,q \in Q$ with $p \neq q$. \medskip

\noindent\textbf{Step 3: Well-posedness of regularization.} Adding the left-hand sides of (\ref{reg.stat.pr}), we obtain
the following nonlinear form defined on $V_\div \times Q$:
$$ \mathbf{a}_n  ( (\vv,p),(\u,q) ) := a(\vv,\u) 
+ d_n(\vv,\u) - b(\u,p) + c_n(p,q) + b(\vv,q), $$
for  $(\vv,p), (\u,q) \in V_\div \times Q$
and a nonlinear operator $\mathcal{A}_n : (V_\div \times Q) \to (V_\div \times Q)'$
defined by
$$\langle \mathcal{A}_n (\vv,p) , (\u,q) \rangle = 
\mathbf{a}_n  ( (\vv,p),(\u,q) ).$$
Since $A_n$, $C_n$, $B$ and $B'$ are continuous,
$\mathcal{A}_n$ is continuous, too.
Furthermore, $\mathcal{A}_n$ is coercive and strictly monotone.
Thus we can apply the theorem of Browder--Minty, see \cite[Theorem 2.18]{roubicek}, to show that for every 
$\mathcal{F} \in (V_\div \times Q)'$ there exists a unique solution 
$(\vv_n , p_n) \in V_\div \times Q$ of the operator equation
$\mathcal{A}_n (\vv_n , p_n) = \mathcal{F}$. \bigskip

\noindent\textbf{Step 4: Uniform bounds.}
Using the second equation of (\ref{reg.stat.pr}) we obtain
$$\begin{aligned} \|\Div(\vv_n)\|_V = \|\Div(\vv_n)\|_{Q'} = 
	\sup_{q \in Q} \frac{| b(\vv_n,q) |}{\|q\|_Q} &= 
	\sup_{q \in Q} \frac{| G(q) - c_n(p_n,q) |}{\|q\|_Q}
	\\&\le \| f \|_V + n \|p_n\|_Q^{1/s}  
\end{aligned}$$
The estimation of $\|\vv_n\|_V$ follows from
the first equation in (\ref{reg.stat.pr}):
$$\begin{aligned}
	C(\beta_0) \|\vv_n\|_V^{s+1} 
	&\le  (\beta |\vv_n|^{s-1}, 
	\vv_n \cdot v_n )_\Omega
	\\&\le a(\vv_n,\vv_n) 
	+ d_n (\vv_n,\vv_n) \\
	&= g(\vv_n) + b(\vv_n,p_n) \\
	&\le  \|g\|_{V'} \|\vv_n\|_V
	+ ( \|g\|_{V'} + \|p_n\|_Q )
	\|\Div(\vv_n)\|_V  .
\end{aligned}$$
Together with the estimate for $\|\Div(\vv_n)\|_V$ 
this yields
\begin{equation}			\label{reg.stat.pr:m-Ls+1-est}
	\begin{aligned}
		\|\vv_n\|_V^{s+1} &\le 
		\frac{1}{C(\beta_0)}
		( \|g\|_{V_\div'} \|\vv_n\|_V
		+ \|g\|_{V_\div'} \|f\|_V
		\\&\qquad + n \|g\|_{V_\div'} \|p_n\|_Q^{1/s}
		+ \|f\|_V \|p_n\|_Q 
		+ n \|p_n\|_Q^{1+1/s} )  .
	\end{aligned}
\end{equation}
To bound $p_n$, we employ the inf-sup condition, that is, there exists $\theta > 0$ such that:
\[
\sup_{\vv \in V_\div} \frac{b(\vv,q)}{\|\vv\|_{V_\div}} \geq \theta \|q\|_Q \quad \forall\, q \in Q.
\]
Together with the first equation in (\ref{reg.stat.pr}) and
the above estimate for $\|\Div(\vv_n)\|_V$, we obtain 
for sufficiently large $n$ ($\frac{1}{n} < \theta^{s/(s+1)}$ is enough)
$$ \|p_n\|_Q^{1/s} \le 
\left( \frac{2 \frac{1}{n^s}}{\theta-\frac{1}{n^{s+1}}} \|f\|_V 
+ \left( \frac{1}{\theta-\frac{1}{n^{s+1}}}
\Big( \|A \vv_n\|_{V_\div'} + \|g\|_{V_\div'} +
n \|f\|_V^s \Big) \right)^{1/s} \right)  . $$
Thus, there exists a constant $C_1 < \infty$, 
independent of $n$, such that
$\|\vv_n\|_V \le C_1$
for $\frac{1}{n} \le \frac{1}{\overline{n}}$.
The bounds for $\|p_n\|_Q$ and $\|\vv_n\|_{V_\div}$ follow. Thus, exist constants $C_{v}, C_p$, 
independent of $n$, such that for sufficiently large $n > 0$
the solution $(\vv_n, p_n)$ of \eqref{reg.stat.pr}
satisfies the following estimates:
\begin{equation}			\label{reg.stat.pr:est}
	\|\vv_n\|_{V_\div} \le C_{v}  , \quad
	\|p_n\|_Q \le C_p  .
\end{equation}

\noindent\textbf{Step 5: Limit process.}
Since $( \vv_n,p_n )_{n \in \N}$ is a bounded sequence
in $V_\div \times Q$, there exists a weakly convergent subsequence, 
again denoted by $( (\vv_n,p_n) )_{n \in \N}$,
with (weak) limit $(\vv,S) \in V_\div \times Q$. 
As it holds
$$\begin{aligned}
	&\| \mathcal{A}(\vv_n,p_n) 
	- \mathcal{F} \|_{(V_\div \times Q)'} \\
	& =  \sup_{(\u,q) \in V_\div \times Q} 
	\frac{ | \mathbf{a} ( (\vv_n,p_n) , (\u,q) )
		- \mathcal{F}(\u,q) |}
	{\|(\u,q)\|_{V_\div \times Q}} \\
	& =  \sup_{(\u,q) \in V_\div \times Q} 
	\frac{ | \mathbf{a}_{n} ( (\vv_n,p_n) , (\u,q) )
		- d_{n}(\vv_n,\u) - c_{n}(p_n,q) 
		- \mathcal{F}(\u,q) |}
	{\|(\u,q)\|_{V_\div \times Q}} \\
	& \le  \frac{1}{n} ( \|\Div(\vv_n)\|_V^s
	+ \| p_n \|_Q^{1/s} ), 
\end{aligned}$$
we deduce that it converges to zero and thus,
the sequence $( \mathcal{A}(\vv_n,p_n) )_{n \in \N}$
converges strongly in $V_\div'$ to $\mathcal{F}$.
Thus, we can conclude that $\mathcal{A} (\vv,p) = \mathcal{F}$ 
in $(V \times Q)'$, i.e.,
$(\vv,p)$ is a solution of (\ref{mix.form.stat.pr}). We note that we obtain the higher regularity of $p \in W^{1+1/s}(\Omega)$ by solving for $\nabla p$.

To show uniqueness, we consider two solutions
$(\vv_1,p_1)$ and $(\vv_2,p_2)$ of (\ref{mix.form.stat.pr}).
Using the test functions $\ww = \vv_1 - \vv_2$
and $q = p_1 - p_2$ we obtain
$$ a(\vv_1,\vv_1-\vv_2) - a(\vv_2,\vv_1-\vv_2)
= \langle A \vv_1 - A \vv_2 , \vv_1-\vv_2 \rangle_ = 0  . $$
Since $A$ is strictly monotone, it follows $\vv_1-\vv_2=0$.  If $\vv \in V_\div$ is given, $p \in Q$ is defined as solution of the variational equation
$b(\u,p) = F(\u) - a(\vv,\u)$ for all $\u \in V_\div$.
Therefore, the uniqueness of $p$ is a direct consequence of the injectivity
of the operator $B': Q \to V_\div'$. 
\qed

\section*{Acknowledgements}
\noindent MF is supported by the state of Upper Austria. AB is supported by the Gutenberg Research College Mainz and by the DFG -- project number 233630050 -- TRR 146 and project number 441153493 -- SPP 2256 ``Variational Methods for Predicting Complex Phenomena in Engineering Structures and Materials".

\bibliographystyle{abbrv}
\bibliography{literature.bib}

\end{document}